\documentclass[reqno,tbtags,a4paper,12pt]{amsart}

\usepackage[bbgreekl]{mathbbol}
\usepackage{amsfonts}

\DeclareSymbolFontAlphabet{\mathbb}{AMSb}
\DeclareSymbolFontAlphabet{\mathbbl}{bbold}

\usepackage{amsmath,amssymb,amsthm,amsfonts,amscd,array,graphicx}
\usepackage[in]{fullpage}

\newcommand{\R}{\mathbb{R}}

\newcommand{\CK}{\mathcal{K}}

\newcommand{\C}{\mathbb{C}}
\newcommand{\Z}{\mathbb{Z}}

\newcommand{\X}{\mathbb{X}}
\newcommand{\V}{\mathbb{V}}

\newcommand{\bS}{\mathbb{S}}

\newcommand{\CV}{\mathcal{V}}

\newcommand{\ba}{\mathbf{a}}

\newcommand{\bx}{\mathbf{x}}
\newcommand{\by}{\mathbf{y}}
\newcommand{\bz}{\mathbf{z}}
\newcommand{\bw}{\mathbf{w}}

\newcommand{\be}{\mathbf{e}}

\newcommand{\tOmega}{\widetilde{\Omega}}

\newcommand{\bell}{\boldsymbol{\ell}}

\newcommand{\Hom}{\mathop{\mathrm{Hom}}\nolimits}

\newcommand{\diam}{\mathop{\mathrm{diam}}\nolimits}
\newcommand{\rank}{\mathop{\mathrm{rank}}\nolimits}
\newcommand{\spn}{\mathop{\mathrm{span}}\nolimits}
\newcommand{\supp}{\mathop{\mathrm{supp}}\nolimits}

\newcommand{\orient}{\mathrm{or}}
\newcommand{\dist}{\mathop{\mathrm{dist}}\nolimits}

\renewcommand{\Re}{\mathop{\mathrm{Re}}\nolimits}

\newtheorem{theorem}{Theorem}[section]
\newtheorem{propos}[theorem]{Proposition}
\newtheorem{cor}[theorem]{Corollary}
\newtheorem{lem}[theorem]{Lemma}

\newtheorem*{bell-conj}{Bellows conjecture}
\theoremstyle{definition}

\newtheorem{defin}[theorem]{Definition}
\newtheorem{remark}[theorem]{Remark}
\numberwithin{equation}{section}

\author{Alexander A. Gaifullin}

\thanks{}

\title{The bellows conjecture for small flexible polyhedra in non-Euclidean spaces}
\thanks{This work is supported by the Russian Science Foundation under grant
14-50-00005}
\date{}

\address{Steklov Mathematical Institute of Russian Academy of Sciences, Gubkina str. 8, Moscow, 119991, Russia}

\email{agaif@mi.ras.ru}

\begin{document}


\begin{abstract}
The \textit{bellows conjecture} claims that the volume of any flexible polyhedron of dimension $3$ or higher is constant during the flexion. The bellows conjecture was proved for flexible polyhedra in the Euclidean spaces $\R^n$, $n\ge 3$, and for bounded flexible polyhedra in the odd-dimensional Lobachevsky spaces~$\Lambda^{2m+1}$, $m\ge 1$. Counterexamples to the bellows conjecture are known in all open hemispheres~$\bS^n_+$, $ n\ge 3$. The aim of this paper is to prove that, nonetheless, the bellows conjecture is true for all flexible polyhedra in either~$\bS^n$ or~$\Lambda^n$, $n\ge 3$, with sufficiently small edge lengths. 
\end{abstract}

\maketitle

\section{Introduction}

Let $\X^n$ be one of the three $n$-dimensional spaces of constant curvature, that is, the Euclidean space~$\R^n$ or the spherical space~$\bS^n$ or the Lobachevsky space~$\Lambda^n$. We shall always normalize the metrics on~$\bS^n$ and~$\Lambda^n$ so that their sectional curvatures are equal to~$1$ and~$-1$, respectively.

A \textit{flexible polyhedron\/} is a connected  oriented closed $(n-1)$-dimensional polyhedral surface~$P$ in~$\X^n$ that admits a continuous deformation~$P_t$ not induced by an isometry of~$\X^n$ such that all faces of~$P_t$ remain congruent to themselves during the deformation. The dihedral angles at $(n-2)$-dimensional faces of~$P_t$ may vary continuously, so it is useful to imagine that $P_t$ has hinges at all $(n-2)$-dimensional faces.
The polyhedral surface~$P$ is allowed to be self-intersecting. Nevertheless, embedded flexible polyhedra, which are sometimes called \textit{flexors\/}, are especially interesting.
A rigorous definition of a flexible polyhedron will be given in Section~\ref{section_flex}.

In dimension~$2$ all generic polygons with at least four sides are flexible. The situation changes drastically in dimensions~$3$ and higher. Examples of flexible polyhedra become rare and not very easy to construct. The first examples of three-dimensional flexible polyhedra were \textit{Bricard's self-intersecting flexible octahedra\/}, see~\cite{Bri97}. Now, examples of self-intersecting flexible polyhedra are known in all spaces of constant curvature of all dimensions $n\ge 3$, see~\cite{Gai13}. The first example of a flexor (an embedded flexible polyhedron) in~$\R^3$ was constructed by Connelly~\cite{Con77}; similar examples exist in~$\Lambda^3$ and~$\bS^3$, cf.~\cite{Kui78}. The author~\cite{Gai15a} constructed examples of embedded flexible polyhedra in the open hemispheres~$\bS^n_+$ for all $n\ge 3$. It still remains unknown if there exist embedded flexible polyhedra in the Euclidean spaces~$\R^n$ or in the Lobachevsky spaces~$\Lambda^n$, where $n\ge 4$.

If the polyhedral surface~$P$ is embedded, then the oriented $n$-dimensional volume of the region bounded by it will be called the \textit{oriented volume\/} of the polyhedron~$P$. For self-intersecting polyhedra, there is a suitable notion of a \textit{generalized oriented volume}, see Section~\ref{section_flex} for details. 

\begin{bell-conj}
Let $\X^n$ be the Euclidean space~$\R^n$ or the Lobachevsky space~$\Lambda^n$ or the open hemisphere~$\bS^n_+$\,,  where $n\ge 3$. Then the generalized oriented volume of any flexible polyhedron in~$\X^n$ is constant during the flexion.
\end{bell-conj}

Originally this conjecture was proposed for flexible polyhedra in~$\R^3$, see~\cite{Con78},~\cite{Kui78}. Hence the above conjecture is sometimes called the \textit{generalized bellows conjecture}.

Studying the bellows conjecture, it is standard to restrict ourselves to considering of \textit{simplicial polyhedra\/}, that is, polyhedra all faces of which are simplices. Indeed, any polyhedral surface has a simplicial  subdivision. Passing to this subdivision, we introduce new hinges. Hence all flexions of the initial polyhedron induce flexions of the obtained simplicial polyhedron, and, possibly, some new flexions appear. Thus the bellows conjecture for simplicial polyhedra would imply the bellows conjecture for arbitrary polyhedra. So further we always consider simplicial polyhedra only. Notice that for simplicial polyhedra, a flexion is a deformation that preserves the combinatorial type and the edge lengths.

The classical bellows conjecture (in~$\R^3$) was proved by Sabitov~\cite{Sab96}--\cite{Sab98b}; another proof was given in~\cite{CSW97}. Recently, the author has proved the bellows conjecture for all flexible polyhedra in the Euclidean spaces~$\R^n$, $n\ge 4$ (see~\cite{Gai11},~\cite{Gai12}), and for all bounded flexible polyhedra in the odd-dimensional Lobachevsky spaces~$\Lambda^{2m+1}$, $m\ge 1$ (see~\cite{Gai15b}). On the other hand, the bellows conjecture is not true for flexible polyhedra in the hemispheres~$\bS^n_+$, $n\ge 3$. The first counterexample in~$\bS^3_+$ was constructed by Alexandrov~\cite{Ale97}. The counterexamples in all dimensions were constructed by the author~\cite{Gai15a}. It is still unknown if the bellows conjecture is true for flexible polyhedra in the even-dimensional Lobachevsky spaces.

The aim of the present paper is to prove that, nonetheless, the bellows conjecture is true for flexible polyhedra in all spheres~$\bS^n$ and in all Lobachevsky spaces~$\Lambda^n$, where $n\ge 3$, provided that the lengths of edges of these flexible polyhedra are small enough.

\begin{theorem}\label{theorem_main}
Let\/ $\X^n$ be either\/~$\bS^n$ or\/~$\Lambda^n,$ $n\ge 3$. Let $P_t$ be a simplicial flexible polyhedron in~$\X^n$ with $m$ vertices such that all edges of~$P_t$ have lengths smaller than $2^{-m^2(n+4)}$. Then the generalized oriented volume of~$P_t$ is constant during the flexion.
\end{theorem}

This result is obtained by joining together the combinatorial approach developed by the author in~\cite{Gai12} to prove the bellows conjecture in the Euclidean spaces of dimensions~$4$ and higher, and the analytic approach developed by the author in~\cite{Gai15b} to prove the bellows conjecture in the odd-dimensional Lobachevsky spaces. Namely, we study the analytic continuation of the volume function to the complexification of the configuration space of polyhedra with the prescribed edge lengths, and then use a quantitative analog of the main combinatorial lemma in~\cite{Gai12} to prove that this analytic continuation is well defined and single-valued.

This paper is organized in the following way. In Section~\ref{section_flex} we give precise definitions of simplicial polyhedra and their generalized oriented volumes, and reformulate Theorem~\ref{theorem_main} in a more rigorous way (Theorem~\ref{theorem_main2}). Section~\ref{section_simplex} contains a survey of results on the analytic continuation of the oriented volume of a simplex in either~$\bS^n$ or~$\Lambda^n$ that is considered as a function of the coordinates of the vertices. The aim of Sections~\ref{section_ordering} and~\ref{section_KGkappa} is to construct certain simplicial complex $\CK(G,\varkappa)$ associated to a symmetric complex matrix~$G$ with units on the diagonal and a positive number~$\varkappa$, and to prove that if the rank of~$G$ does not exceed~$2r$, then for certain~$\varkappa$, the complex $\CK(G,\varkappa)$ collapses on a subcomplex of dimension less than~$r$. This assertion (Theorem~\ref{theorem_KCkappa}) is formulated and proved in Section~\ref{section_KGkappa}. Section~\ref{section_ordering} contains some preliminary definitions and results concerning collapses of simplicial complexes and special orderings of simplices of simplicial complexes, which we name \textit{hereditary orderings}. Theorem~\ref{theorem_KCkappa} is the main technical assertion in our proof of Theorem~\ref{theorem_main2}. However, it seems that it may be interesting in itself. The proof of Theorem~\ref{theorem_main2} is given in Section~\ref{section_proof}.

\section{Polyhedra and their volumes}\label{section_flex}

\subsection{Simplicial polyhedra}
We shall conveniently work with the most general concept of a simplicial polyhedron, which was called in~\cite{Gai11} a \textit{cycle polyhedron}. The idea of this definition is to mean under a closed $(n-1)$-dimensional polyhedral surface an arbitrary image of an $(n-1)$-dimensional simplicial cycle that is allowed to be degenerate or self-intersecting.
First, we need to introduce some background material concerning simplicial complexes.

Recall that an \textit{\textnormal{(}abstract\textnormal{)} simplicial complex\/} on a vertex set~$S$ is a set $K$ of subsets of~$S$ that satisfies the following conditions
\begin{itemize}
\item $\emptyset\in K$, 
\item if $\tau\in K$ and $\sigma\subseteq\tau$, then $\sigma\in K$.
\end{itemize}
We shall usually assume that the simplicial complex has no \textit{fictive vertices}, that is, all one-element subsets of~$S$ belong to~$K$.
By definition, $\dim\sigma$ is the cardinality of~$\sigma$ minus~$1$. 
The geometric realization of a simplicial complex $K$ will be denoted by~$|K|$. 

Let $C_*(K)$ be the simplicial chain complex of~$K$ with integral coefficients, and let $\partial$ be the differential of this complex. Recall that $C_k(K)$ is the free Abelian group generated by oriented $k$-dimensional simplices of~$K$. (Reversing the orientation of a simplex changes the sign of the corresponding generator of~$C_k(K)$.) We denote by $[u_0,\ldots,u_k]$ the oriented simplex with vertices $u_0,\ldots,u_k$ with the orientation given by the indicated order of vertices. Then 
$$
\partial [u_0,\ldots,u_k]=\sum_{j=0}^k(-1)^j[u_0,\ldots,\widehat{u}_j,\ldots,u_k],
$$
where hat denotes the omission of the corresponding letter. A chain $\xi\in C_k(K)$ is called a \textit{cycle} if $\partial\xi=0$. The group of $k$-dimensional cycles in~$K$ will be denoted by~$Z_k(K)$. The \textit{support\/} of a chain $\xi\in C_k(K)$ is the subcomplex $\supp\xi\subseteq K$ consisting of all $k$-simplices that enter~$\xi$ with non-zero coefficients, and all their subsimplices.

We denote by $C^*(K;A)$ the simplicial cochain complex of~$K$ with coefficients in an Abelian group~$A$, i.\,e., the complex consisting of the groups $C^k(K;A)=\Hom(C_k(K),A)$ with respect to the differential~$\delta$ adjoint to~$\partial$. A cochain~$f$ satisfying $\delta f=0$ is called a \textit{cocycle}. We denote by~$Z^k(K;A)\subseteq C^k(K;A)$ the subgroup consisting of all cocycles.

We denote by~$[m]$ the set $\{1,\ldots,m\}$. We denote by~$\Delta_{[m]}$ the simplex on the vertex set~$[m]$, i.\,e., the simplicial complex consisting of all subsets $\sigma\subseteq[m]$.

\begin{defin}
Suppose that $\xi\in Z_{n-1}(\Delta_{[m]})$. A mapping $P\colon |\supp\xi|\to\R^n$ is called a \textit{simplicial polyhedron\/} of combinatorial type~$\xi$ if the restriction of~$P$ to every simplex of~$|\supp\xi|$ is affine linear. 
\end{defin}

To give a similar definition of a polyhedron in~$\bS^n$ or in~$\Lambda^n$, we need to introduce the concept of a \textit{pseudo-linear\/} mapping. Recall the standard vector models for~$\bS^n$ and~$\Lambda^n$. For our further considerations, it will be convenient for us to place these two vector models into the same complex vector space~$\C^{n+1}$.

Consider the space $\C^{n+1}$ with the standard basis $\be_0,\ldots,\be_n$. Endow $\C^{n+1}$ with the standard bilinear scalar product
$$
\langle \bz,\bw\rangle=\bz^t\bw=z_0w_0+\cdots+z_nw_n,
$$ 
where  $\bz=(z_0,\ldots,z_n)^t$ and $\bw=(w_0,\ldots,w_n)^t$, and $t$ denotes the transpose of a matrix. 

Let $Q\subset \C^{n+1}$ be the quadric given by
$$
\langle \bz,\bz\rangle=1.
$$

We put,
\begin{gather*}
\R^{n+1}=\spn_{\R}(\be_0,\be_1,\ldots,\be_n),\qquad
\R^{1,n}=\spn_{\R}(\be_0,i\be_1,\ldots,i\be_n).
\end{gather*} 
Then the scalar product $\langle\cdot,\cdot\rangle$ restricted to either of the spaces~$\R^{n+1}$ and~$\R^{1,n}$ is real-valued, and $\R^{n+1}$ and~$\R^{1,n}$ with this scalar product are the $(n+1)$-dimensional Euclidean space and the pseudo-Euclidean space of signature $(1,n)$, respectively. The intersection $\bS^n=Q\cap\R^{n+1}$ is the standard unit $n$-dimensional sphere.  The intersection $Q\cap\R^{1,n}$ is the $n$-dimensional double-sheeted hyperboloid. The sheet of this hyperboloid contained in the half-space $x_0>0$ serves as the standard vector model for the $n$-dimensional Lobachevsky space~$\Lambda^n$. Further we shall always identify $\Lambda^n$ with this sheet of the hyperboloid $Q\cap\R^{1,n}$. We endow $\bS^n$ and $\Lambda^n$ with the orientations so that $\be_1,\ldots,\be_n$ and $i\be_1,\ldots,i\be_n$ are positively oriented bases of the tangent spaces~$T_{\be_0}\bS^n$ and~$T_{\be_0}\Lambda^n$, respectively.

If $\bz\in\C^{n+1}$ is a vector satisfying $\langle\bz,\bz\rangle>0$, then we put
\begin{equation*}
\mathop{\mathrm{norm}}(\bz)=\frac{\bz}{\sqrt{\langle\bz,\bz\rangle}}\,.
\end{equation*}

Let $\Delta^k$ be an affine simplex with vertices $v_0,\ldots,v_k$, and let $\X^n$ be either~$\bS^n$ or~$\Lambda^n$. A mapping $f\colon\Delta^k\to\X^n$ is called \textit{pseudo-linear\/} if
$$
f(\beta_0v_0+\cdots+\beta_kv_k)=\mathop{\mathrm{norm}}\bigl(\beta_0f(v_0)+\cdots+\beta_kf(v_k)\bigr)
$$
for all $\beta_0,\ldots,\beta_k\ge 0$ such that $\beta_0+\cdots+\beta_k=1$. Obviously, the restriction of a pseudo-linear mapping to a face of~$\Delta^k$ is also a pseudo-linear mapping.

\begin{remark}
Consider any points $\bx_0,\ldots,\bx_k\in \X^n$. In the spherical case we additionally assume that $\bx_p\ne -\bx_q$ for all~$p$ and~$q$. Then  for any $\beta_0,\ldots,\beta_k\ge 0$ such that $\beta_0+\cdots+\beta_k=1$, the vector $\bx=\beta_0\bx_0+\cdots+\beta_k\bx_k$ is proportional to a vector in~$\X^n$ with a positive coefficient. Hence the vector $\mathop{\mathrm{norm}}(\bx)$ belongs to~$\X^n$. Thus there exists a unique pseudo-linear mapping $f\colon\Delta^k\to\X^n$ such that $f(v_j)=\bx_j$ for all~$j$. 
\end{remark}

\begin{defin}
Let $\X^n$ be either~$\bS^n$ or~$\Lambda^n$. Suppose that $\xi\in Z_{n-1}(\Delta_{[m]})$. A mapping $P\colon |\supp\xi|\to\X^n$ is called a \textit{simplicial polyhedron\/} of combinatorial type~$\xi$ if the restriction of~$P$ to every simplex of~$|\supp\xi|$ is pseudo-linear. 
\end{defin}

If we take for $\xi$ the fundamental class of an oriented $(n-1)$-dimensional pseudo-manifold~$K$, we arrive to a more intuitive concept of a simplicial polyhedron: A simplicial polyhedron of combinatorial type~$K$ is a mapping $P\colon |K|\to\X^n$ such that the restriction of~$P$ to every simplex of~$K$ is affine linear (if $\X^n=\R^n$) or pseudo-linear (if $\X^n$ is either~$\bS^n$ or~$\Lambda^n$). If $P$ is an embedding, then this polyhedron is called \textit{embedded}. 

\begin{defin}\label{defin_flex}
Let $\X^n$ be one of the three spaces~$\R^n$, $\bS^n$, and~$\Lambda^n$. Suppose that $\xi\in Z_{n-1}(\Delta_{[m]})$, and $\supp\xi$ is connected. A one-parametric continuous family of mappings $P_t\colon|\supp\xi|\to\X^n$ is called a \textit{flexible polyhedron\/} of combinatorial type~$\xi$ if $P_t$ are simplicial polyhedra of combinatorial type~$\xi$ for all~$t$, and $P_{t_1}$ cannot be obtained from $P_{t_2}$ by an isometry of~$\X^n$ for all sufficiently close to each other $t_1\ne t_2$.
\end{defin}

The condition that $\supp\xi$ must be connected is very natural. Indeed, otherwise we would have two polyhedra that are allowed to move independently, which would not correspond to our intuitive concept of a flexible polyhedron. Usually, some additional conditions of non-degeneracy are imposed on a flexible polyhedron $P_t\colon|\supp\xi|\to\X^n$, cf.~\cite[Sect.~7]{Gai15b}. We do not want to impose any additional conditions, since our main result (Theorem~\ref{theorem_main}) remains true without them.

\subsection{Generalized oriented volume}
Consider a polyhedron $P\colon |\supp\xi|\to\X^n$ of combinatorial type~$\xi$, where $\xi\in Z_{n-1}(\Delta_{[m]})$, and put $K=\supp\xi$.

First, suppose that $\X^n$ is either~$\R^n$ or~$\Lambda^n$. For any point~$\bx\in \X^n\setminus P(|K|)$, we denote by~$\lambda_{\infty}(\bx)$ the algebraic intersection number of a curve going from~$\bx$ to infinity and the $(n-1)$-dimensional cycle $P(\xi)$. (Obviously, this algebraic intersection number is independent of the choice of the curve.) Then $\lambda_{\infty}(\bx)$ is 
an almost everywhere defined piecewise constant function on~$\X^n$ with compact support.
By definition, the \textit{generalized oriented volume\/} of the polyhedron~$P$ is given by
\begin{equation*}
\CV_{\xi}(P)=\int_{\X^n}\lambda_{\infty}(\bx)\,dV_{\X^n},
\end{equation*}
where $dV_{\X^n}$ is the standard volume element in~$\X^n$.

Second, suppose that $\X^n=\bS^n$. This case is somewhat more complicated, since there is no infinity in~$\bS^n$. One way to handle this difficulty is to consider the generalized oriented volume as an element of the group $\R/\sigma_n\Z$, where $\sigma_n$ is the volume of~$\bS^n$. Choose  a point~$\bx_0\in\bS^n\setminus P(|K|)$. For any point~$\bx\in \bS^n\setminus P(|K|)$, we denote by~$\lambda_{\bx_0}(\bx)$ the algebraic intersection number of a curve going from~$\bx$ to~$\bx_0$ and the $(n-1)$-dimensional cycle $P(\xi)$.  By definition, the \textit{generalized oriented volume\/} of the polyhedron~$P$ is given by
\begin{equation*}
\CV_{\xi}(P)=\int_{\bS^n}\lambda_{\bx_0}(\bx)\,dV_{\bS^n}\pmod{\sigma_n\Z}.
\end{equation*}
If we replace~$\bx_0$ with any other point~$\bx_0'\in \bS^n\setminus P(|K|)$, then we shall easily see that $\lambda_{\bx_0'}(\bx)=\lambda_{\bx_0}(\bx)+\lambda_{\bx_0'}(\bx_0)$ for all~$\bx$, hence, the value of the integral in the right-hand side will change by~$\lambda_{\bx_0'}(\bx_0)\sigma_n$, which belongs to~$\sigma_n\Z$, since $\lambda_{\bx_0'}(\bx_0)\in\Z$.

Nevertheless, it will be important for us that, for a spherical polyhedron of a sufficiently small diameter, one can define its generalized oriented volume so that it will be a well-defined real number. Namely, assume that $\diam P(|K|)<\pi/2$. Then we define the \textit{generalized oriented volume\/} of~$P$ by
\begin{equation*}
\CV_{\xi}(P)=\int_{\bS^n}\lambda_{-\by}(\bx)\,dV_{\bS^n},
\end{equation*}
where $\by$ is an arbitrary point in~$P(|K|)$. If $\by,\by'\in P(|K|)$, then the geodesic segment with endpoints~$-\by$ and~$-\by'$ does not intersect~$P(|K|)$, hence, $\lambda_{-\by'}(\bx)=\lambda_{-\by}(\bx)$ for all~$\bx$. Therefore the real number $\CV_{\xi}(P)$ is well defined.

Equivalently, the generalized oriented volume of~$P$ can be determined in the following way. For any points $\ba_0,\ldots,\ba_n\in \X^n$, we denote by $V_{\X^n}^{\orient}(\ba_0,\ldots,\ba_n)$ the oriented volume of the simplex in~$\X^n$ with vertices $\ba_0,\ldots,\ba_n$. In the spherical case, this oriented volume is  defined if and only if no two points~$\ba_j$ and~$\ba_k$ are antipodal to each other. 

 If $n\ge 2$, then a cycle $\xi\in  Z_{n-1}(\Delta_{[m]})$ is a boundary. 

\begin{lem}
Let $\X^n$ be one of the three spaces~$\R^n$, $\bS^n$, and~$\Lambda^n$, $n\ge 2$. If $\X^n=\bS^n$, then we assume that $\diam P(|\supp\xi|)<\pi/2$.
Consider a chain $\eta\in C_{n}(\Delta_{[m]})$ such that $\partial\eta=\xi$. Suppose that $\eta=\sum_{q=1}^Nc_q[u_{q,0},\ldots,u_{q,n}]$. Then 
\begin{equation}\label{eq_or_vol}
\CV_{\xi}(P)=\sum_{q=1}^N c_qV_{\X^n}^{\orient}(\ba_{u_{q,0}},\ldots,\ba_{u_{q,n}}),
\end{equation}
where $\ba_u=P(u)$, $u=1,\ldots,m$.
\end{lem}  

\begin{proof}
Recall that a pseudo-linear (or affine linear) mapping $P\colon|\supp\xi|\to\X^n$ can be uniquely restored from the images of the vertices. Let us fix a mapping $P\colon[m]\to\X^n$, and put $\ba_u=P(u)$, $u=1,\ldots,m$. In the case $\X^n=\bS^n$, this mapping should additionally satisfy $\dist_{\bS^n}(\ba_u,\ba_v)<\pi/2$ for all $u,v\in[m]$. Then both sides of~\eqref{eq_or_vol} can be considered as functions of a chain $\eta\in C_{n}(\Delta_{[m]})$. Obviously, the right-hand side of~\eqref{eq_or_vol} is a linear funcion of~$\eta$. Let us prove that the left-hand side of~\eqref{eq_or_vol}, i.\,e., $\CV_{\partial\eta}(P)$ is also a linear function of~$\eta$.  Assume that $\eta=\eta^{(1)}+\eta^{(2)}$. We put $\by=\infty$ if $\X^n$ is either~$\R^n$ or~$\Lambda^n$, and $\by=-\ba_1$ if $\X^n=\bS^n$. Let $\lambda_{\by}(\bx)$, $\lambda_{\by}^{(1)}(\bx)$, and $\lambda_{\by}^{(2)}(\bx)$ be the algebraic intersection numbers of a curve from~$\bx$ to~$\by$ with the cycles~$P(\partial\eta)$, $P(\partial\eta^{(1)})$, and $P(\partial\eta^{(2)})$, respectively. Then $\lambda_{\by}(\bx)=\lambda_{\by}^{(1)}(\bx)+\lambda_{\by}^{(2)}(\bx)$ for all~$\bx$, which immediately implies that $\CV_{\partial\eta}(P)=\CV_{\partial\eta^{(1)}}(P)+\CV_{\partial\eta^{(2)}}(P)$. Thus $\CV_{\partial\eta}(P)$ is a linear function of~$\eta$. 

Therefore it is sufficient to prove~\eqref{eq_or_vol} for $\eta=[u_0,\ldots,u_n]$. Then $P(\partial\eta)$ is the oriented boundary of the simplex $[\ba_{u_0},\ldots,\ba_{u_n}]$. If this simplex is degenerate, then both sides of~\eqref{eq_or_vol} vanish. Assume that the simplex $[\ba_{u_0},\ldots,\ba_{u_n}]$ is non-degenerate, and put $\varepsilon=1$ if this simplex is positively oriented and $\varepsilon=-1$ if this simplex is negatively oriented. Then $\lambda_{\by}(\bx)=\varepsilon$ for $\bx$ inside the simplex $[\ba_{u_0},\ldots,\ba_{u_n}]$, and $\lambda_{\by}(\bx)=0$ for $\bx$ outside the simplex $[\ba_{u_0},\ldots,\ba_{u_n}]$. Hence, 
$$
\CV_{\partial\eta}(P)=\varepsilon V_{\X^n}(\ba_{u_0},\ldots,\ba_{u_n})=V^{\orient}_{\X^n}(\ba_{u_0},\ldots,\ba_{u_n}),
$$
which is exactly~\eqref{eq_or_vol}.
\end{proof}

\subsection{Configuration spaces of polyhedra}
It is well known that the distances in $\bS^n$ and~$\Lambda^n$ satisfy
\begin{align*}
\cos\dist_{\bS^n}(\bx,\by)&=\langle \bx,\by\rangle,&\bx,\by&\in \bS^n,\\
\cosh\dist_{\Lambda^n}(\bx,\by)&=\langle \bx,\by\rangle,&\bx,\by&\in\Lambda^n.
\end{align*}

Suppose that $\xi\in Z_{n-1}(\Delta_{[m]})$,  and $K=\supp\xi$ is connected.  A polyhedron $P\colon K\to\X^n$ of combinatorial type~$\xi$ is determined solely by the set of its vertices $P(u)=\ba_u$, $u=1,\ldots,m$. We shall always identify a polyhedron~$P$ with the corresponding point $A=(\ba_1,\ldots,\ba_m)\in(\X^n)^m$.

Let $s$ be the number of edges of~$K$. Let $\bell=(\ell_{uv})$, $\{u,v\}\in K$, be the prescribed set of edge lengths of~$P$. Then the configuration space of all polyhedra~$P$ in~$\X^n$ of combinatorial type~$\xi$ with the prescribed set of edge lengths~$\bell$ is the subset $\Sigma(\xi,\X^n,\bell)\subset(\X^n)^m$ consisting of all points $A=(\ba_1,\ldots,\ba_m)$, $\ba_u\in\X^n$, such that  $\dist_{\X^n}(\ba_u,\ba_v)=\ell_{uv}$ for all $\{u,v\}\in K$. In other words, the subset $\Sigma(\xi,\X^n,\bell)\subset(\X^n)^m$ is given by~$s$ polynomial equations 
\begin{equation}\label{eq_polynom}
\langle\ba_u,\ba_v\rangle=\left\{
\begin{aligned}
&\cos\ell_{uv}&&\text{if}\ \ \X^n=\bS^n,\\
&\cosh\ell_{uv}&&\text{if}\ \ \X^n=\Lambda^n,
\end{aligned}
\right.\qquad\qquad \{u,v\}\in K.
\end{equation}

If  $\X^n=\Lambda^n$, then the generalized oriented volume $\CV_{\xi}(P)$ is well defined for any polyhedron $P\colon|K|\to\Lambda^n$. Identifying $P$ with the corresponding point~$A\in\Sigma(\xi,\Lambda^n,\bell)$, we obtain a well-defined function~$\CV_{\xi}(A)$ on~$\Sigma(\xi,\Lambda^n,\bell)$.
  
If $\X^n=\bS^n$, then  the generalized oriented volume $\CV_{\xi}(P)$ is well defined whenever $\diam P(|K|)<\pi/2$. Assume that  $\ell_{uv}<\pi/(2(m-1))$ for all $\{u,v\}\in K$. Then the inequality $\diam P(|K|)<\pi/2$ holds for all polyhedra $P\colon|K|\to\bS^n$ with the set of edge lengths~$\bell$. Hence  $\CV_{\xi}(A)$ is a well-defined function on~$\Sigma(\xi,\bS^n,\bell)$.
  
A flexible polyhedron $P_t\colon |K|\to\X^n$ of combinatorial type~$\xi$ with the set of edge lengths~$\bell$ is nothing but a continuous path in~$\Sigma(\xi,\X^n,\bell)$. Hence the following theorem is a  reformulation of Theorem~\ref{theorem_main}.

\begin{theorem}\label{theorem_main2}
Suppose that $n\ge 3$ and $\X^n$ is either~$\bS^n$ or~$\Lambda^n$. Suppose that
$\xi\in Z_{n-1}(\Delta_{[m]})$ and $K=\supp\xi$ is connected. Let $\bell=(\ell_{uv})_{\{u,v\}\in K}$ be the prescribed set of edge lengths such that $0\le \ell_{uv}<2^{-m^2(n+4)}$ for all $\{u,v\}\in K$. Then the function~$\CV_{\xi}(A)$ is constant on every connected component of\/~$\Sigma(\xi,\X^n,\bell)$.
\end{theorem}  

\begin{remark}
It is easy to see that $2^{-m^2(n+4)}<\pi/(2(m-1))$ whenever $m\ge 2$. Therefore under the hypothesis of Theorem~\ref{theorem_main2} the function $\CV_{\xi}(A)$ is well defined on~$\Sigma(\xi,\X^n,\bell)$.  
\end{remark}
  
\section{Analytic continuation of the volume of a simplex} \label{section_simplex}

Let $\X^n$ be either $\bS^n$ or~$\Lambda^n$. To treat both cases together, we shall conveniently introduce a parameter~$\nu$ such that $\nu=1$ if $\X^n=\bS^n$ and $\nu=i$ if $\X^n=\Lambda^n$.

Let $\Delta^n$ be a simplex in~$\X^n$ with vertices $\ba_0,\ldots,\ba_n$, $\ba_j=(a_{0j},\ldots, a_{nj})^t$. Then $A=(a_{jk})$ is the matrix with columns $\ba_0,\ldots,\ba_n$. Let $G=(g_{jk})$ be the Gram matrix of the vectors $\ba_0,\ldots,\ba_n$. Then $G=A^tA$,  $g_{jj}=1$, $j=0,\ldots,n$, and $g_{jk}$ are the cosines of the edge lengths of~$\Delta^n$ if $\X^n=\bS^n$ and the hyperbolic cosines of the edge lengths of~$\Delta^n$ if $\X^n=\Lambda^n$. Let $K(\Delta^n)\subset \C^{n+1}$ be the convex cone generated by the vectors~$\ba_0,\ldots,\ba_n$, that is, the set consisting of all linear combinations $t_0\ba_0+\cdots+t_n\ba_n$, where $t_0,\ldots,t_n\ge 0$. It is well known that the volume of the simplex~$\Delta^n$ can be computed in the following way, see \cite{Cox35}, \cite{Aom77}, \cite{AVS88}:
\begin{multline*}
V_{\X^n}(\Delta^n)=\left(\int\limits_0^{+\infty}t^ne^{-t^2}\right)^{-1}\int\limits_{K(\Delta^n)}e^{-\langle\bx,\bx\rangle}d\bx\\
{}=\frac{2\sqrt{|\det G|}}{\Gamma\left(\frac{n+1}{2}\right)}\int\limits_{t_0,\ldots,t_n\ge 0}\exp\left(-\sum_{j,k=0}^ng_{jk}t_jt_k\right)dt_0\cdots dt_n,
\end{multline*}
where $d\bx$ is the standard volume element in~$\R^{n+1}$ (if $\X^n=\bS^n$) or in~$\R^{1,n}$ (if $\X^n=\Lambda^n$).

We have $\det G=(\det A)^2$. It is easy to see that, for any simplex~$\Delta^n$ in~$\X^n$, the number $\nu^{-n}\det A$ is real, $\nu^{-n}\det A>0$  if $\Delta^n$ is positively oriented, and  $\nu^{-n}\det A<0$ if $\Delta^n$ is negatively oriented. Therefore the oriented volume of~$\Delta^n$ is given by
\begin{equation}\label{eq_Vor}
V_{\X^n}^{\orient}(\Delta^n)=\frac{2\det A}{\nu^n\,\Gamma\left(\frac{n+1}{2}\right)}\int\limits_{t_0,\ldots,t_n\ge 0}\exp\left(-\sum_{j,k=0}^n\langle\ba_j,\ba_k\rangle t_jt_k\right)dt_0\cdots dt_n.
\end{equation}

Let $\Psi\subset Q^{n+1}\subset\C^{(n+1)^2}$ be the subset consisting of all points $A=(\ba_0,\ldots,\ba_n)$, $\ba_j\in Q$, such that
$$
\Re\sum_{j,k=0}^n\langle\ba_j,\ba_k\rangle t_jt_k>0
$$
for all $t_0,\ldots,t_n\ge 0$ provided that at least one of the numbers~$t_j$ is strictly positive. Obviously, $\Psi$ is a relatively open subset of $Q^{n+1}$. It is easy to see that the function
\begin{equation}\label{eq_F(A)}
F(A)=\det A\int\limits_{t_0,\ldots,t_n\ge 0}\exp\left(-\sum_{j,k=0}^n\langle\ba_j,\ba_k\rangle t_jt_k\right)dt_0\cdots dt_n
\end{equation}
is well defined and holomorphic on~$\Psi$.

Obviously, the set~$\Psi$ contains all points $A=(\ba_0,\ldots,\ba_n)\in Q^{n+1}$ such that their Gram matrices $G=A^tA$ have positive definite real parts. On the other hand,  $\Psi$ contains all points $(\ba_0,\ldots,\ba_n)\in Q^{n+1}$ such that $\Re\langle\ba_j,\ba_k\rangle>0$ for all~$j$ and~$k$. In particular, it follows that $\Psi$ contains all points $(\ba_0,\ldots,\ba_n)$ such that $\ba_j\in\Lambda^n$ for all~$j$ and contains all points $(\ba_0,\ldots,\ba_n)$ such that $\ba_j\in\bS^n$ for all~$j$ and no two points~$\ba_j$ and~$\ba_k$ are antipodal to each other. Formula~\eqref{eq_Vor} yields that 
\begin{equation}\label{eq_orient_vol_F}
V_{\X^n}^{\orient}(\ba_0,\ldots,\ba_n)=\frac{2}{\nu^n\,\Gamma\left(\frac{n+1}{2}\right)}\,F(\ba_0,\ldots,\ba_n)
\end{equation}
whenever $\ba_0,\ldots,\ba_n\in\X^n$.

We do not know how to describe the geometry of the set~$\Psi$ effectively. So we shall conveniently restrict ourselves to a smaller subset, which is much more handable.  Namely, let $\Omega\subset Q^{n+1}\subset\C^{(n+1)^2}$ be the subset consisting of all points $(\ba_0,\ldots,\ba_n)$, $\ba_j\in Q$, such that
$$
|\langle\ba_j,\ba_k\rangle-1|<1
$$
for all $j$ and~$k$. Obviously, $\Omega$ is a relatively open subset of $Q^{n+1}$ and $\Omega\subset\Psi$. Hence the function~$F$ is holomorphic on~$\Omega$. Besides, the function~$F$ restricted to~$\Omega$ is bounded. Indeed, for all $A\in \Omega$, we have,
\begin{multline}\label{eq_est_F}
|F(A)|< \sqrt{|\mathbf{g}_0|\cdots|\mathbf{g}_n|}\int\limits_{t_0,\ldots,t_n\ge 0}\exp\left(-\sum_{j=0}^nt_j^2\right)dt_0\cdots dt_n\\
{}<\left(2\sqrt{n+1}\right)^{\frac{n+1}2}\left(\frac{\sqrt{\pi}}{2}\right)^{n+1}=
\left(\frac{\pi^2(n+1)}{4}\right)^{\frac{n+1}4},
\end{multline} 
where $\mathbf{g}_0,\ldots,\mathbf{g}_n$ are the columns of the matrix~$G=A^tA$, and  $|\mathbf{g}|=\sqrt{\overline{\mathbf{g}}^t\mathbf{g}}$ is the Hermitian length of a vector $\mathbf{g}\in\C^{n+1}$.

\begin{lem}\label{lem_zero_sum}
Suppose that $\ba_0,\ldots,\ba_{n+1}$ are points in~$Q$ such that\/ $|\langle\ba_j,\ba_k\rangle-1|<1$ for all\/~$j$ and~$k$. Then 
\begin{equation}\label{eq_stand_relation}
\sum_{j=0}^{n+1}(-1)^jF(\ba_0,\ldots,\hat\ba_j,\ldots,\ba_{n+1})=0.
\end{equation}
\end{lem}

To prove this assertion, we need the following auxiliary lemma.

\begin{lem}\label{lem_connect}
Suppose $\tOmega$ is the subset of~$Q^{n+2}$ consisting of all points $(\ba_0,\ldots,\ba_{n+1}),$ $\ba_j\in Q,$ such that
$|\langle\ba_j,\ba_k\rangle-1|<1$
for all $j$ and~$k$. Then~$\tOmega$ is pathwise connected. 
\end{lem}

\begin{proof}
Let $\widetilde{A}^0=(\ba^0_0,\ldots,\ba^0_{n+1})$ be a point in~$\tOmega$. Put $g_{jk}=\langle\ba_j^0,\ba_k^0\rangle$ and $q_{jk}=g_{jk}-1$. Consider the path~$\widetilde{A}(t)=(\ba_0(t),\ldots,\ba_{n+1}(t))$, $t\in[0,1]$, in~$(\C^{n+1})^{n+2}$ parametrized by
%\begin{equation}\label{eq_path}
\begin{align*}
\ba_0(t)&=\ba^0_0,\\
\ba_j(t)&=\frac
{t\left(1+\frac12q_{0j}t\right)\ba_0^0+(1-t)\,\ba_j^0}
{1+q_{0j}\left(t-\frac{1}{2}t^2\right)}\,
,&&j=1,\ldots,n+1.
\end{align*}
%\end{equation}
Since $|q_{0j}|<1$, the denominator of the latter expression does not vanish for $t\in[0,1]$.
A direct computation shows that $\langle\ba_j(t),\ba_j(t)\rangle=1$, hence, $\widetilde{A}(t)\in Q^{n+2}$ for all $t\in[0,1]$. Also, a direct computation shows that 
\begin{align*}
\langle\ba_0(t),\ba_j(t)\rangle-1&=\frac{q_{0j}(1-t)^2}{1+q_{0j}\left(t-\frac{1}{2}t^2\right)}\,,&&1\le j\le n+1,\\
\langle\ba_j(t),\ba_k(t)\rangle-1&=\frac{q_{jk}(1-t)^2}{\left(1+q_{0j}\left(t-\frac{1}{2}t^2\right)\right)\left(1+q_{0k}\left(t-\frac{1}{2}t^2\right)\right)}\,,&&1\le j<k\le n+1.
\end{align*}
Since $|q_{0j}|<1$ and $0\le t\le 1$, we obtain that $|1+q_{0j}\left(t-\frac{1}{2}t^2\right)|\ge 1-t$, $j=1,\ldots,n+1$. Hence, 
$$
|\langle\ba_j(t),\ba_k(t)\rangle-1|\le |q_{jk}|<1,\qquad 0\le j<k\le n+1.
$$
Therefore the path $\widetilde{A}(t)$ is contained in~$\tOmega$. This path connects the initial point $\widetilde{A}(0)=\widetilde{A}^0$ with the point $\widetilde{A}(1)=(\ba^0_0,\ldots,\ba^0_0)$. It is easy to see that the hyperboloid~$Q$ is connected and the points $(\ba,\ldots,\ba)$ belong to~$\tOmega$ for all $\ba\in Q$. Hence any two points of the form $(\ba,\ldots,\ba)$, $\ba\in Q$, can be connected by a path in~$\tOmega$. Thus the set~$\tOmega$ is pathwise connected.\end{proof}

\begin{proof}[Proof of Lemma~\ref{lem_zero_sum}]
The set $\tOmega$ is a relatively open subset of the smooth complex affine variety~$Q^{n+2}$. Hence $\tOmega$ is a complex analytic manifold of dimension~$n(n+2)$.
Let $H(\ba_0,\ldots,\ba_{n+1})$ be the left-hand side of~\eqref{eq_stand_relation}. Since $F$ is a holomorphic function on~$\Omega$, we see that $H$ is a holomorphic function on~$\tOmega$.

If $\ba_0,\ldots,\ba_{n+1}\in\Lambda^n$, then the simplex in~$\Lambda^n$ with vertices $\ba_0,\ldots,\ba_{n+1}$ is degenerate. Hence the sum of oriented volumes of its $n$-dimensional faces is equal to zero:
$$
\sum_{j=0}^{n+1}(-1)^jV^{\orient}_{\Lambda^n}(\ba_0,\ldots,\hat\ba_j,\ldots,\ba_{n+1})=0.
$$
By~\eqref{eq_orient_vol_F}, we obtain that $H(\ba_0,\ldots,\ba_{n+1})=0$. 

%Obviously, the set $\tOmega$ contains all points $(\ba_0,\ldots,\ba_{n+1})\in(\Lambda^n)^{n+2}$ that are  close enough to the point $(\be_0,\ldots,\be_0)$. 
The mapping 
$$
\varphi(z_0,\ldots,z_n)=\left(\frac{z_1}{iz_0},\ldots,\frac{z_n}{iz_0}\right)
$$
provides a biholomorphic isomorphism of a (connected) neighborhood~$U$ of the point~$\be_0$ in~$Q$ onto a neighborhood~$W$ of the origin in~$\C^n$ such that $\varphi(U\cap\Lambda^n)=W\cap\R^n$. Hence the mapping~$\Phi=\varphi^{n+2}$ is a biholomorphic isomorphism of the neighborhood~$\mathcal{U}=U^{n+2}$ of the point~$(\be_0,\ldots,\be_0)$ in~$Q^{n+2}$ onto the neighborhood~$\mathcal{W}=W^{n+2}$ of the origin in~$\C^{n(n+2)}$ such that $\Phi(\mathcal{U}\cap(\Lambda^n)^{n+2})=\mathcal{W}\cap \R^{n(n+2)}$. We may assume that the neighborhood $U$ is chosen small enough to obtain that $\mathcal{U}\subseteq\tOmega$. Any holomorphic function on~$\mathcal{W}$ that vanishes identically on~$\mathcal{W}\cap \R^{n(n+2)}$ vanishes identically on the whole~$\mathcal{W}$. Therefore, any holomorphic function on~$\mathcal{U}$ that vanishes identically on~$\mathcal{U}\cap (\Lambda^n)^{n+2}$ vanishes identically on the whole~$\mathcal{U}$. In particular, $H(\ba_0,\ldots,\ba_{n+1})=0$ for all $(\ba_0,\ldots,\ba_{n+1})\in\mathcal{U}$.
%Hence the set $\tOmega\cap (\Lambda^n)^{n+2}$ is a smooth real analytic submanifold of the complex analytic manifold~$\tOmega$, and $\dim_{\R}(\tOmega\cap (\Lambda^n)^{n+2})=\dim_{\C}\tOmega=n(n+2)$. Besides,  the tangent space  $T_{(\be_0,\ldots,\be_0)}(\tOmega\cap (\Lambda^n)^{n+2})=(T_{\be_0}\Lambda^n)^{n+2}$ generates over~$\C$ the whole tangent space $T_{(\be_0,\ldots,\be_0)}\tOmega=(T_{\be_0}Q)^{n+2}$. Therefore the pair $(\tOmega,\tOmega\cap(\Lambda^n)^{n+2})$ at~$(\be_0,\ldots,\be_0)$ is locally biholomorphic  to the pair $(\C^{n(n+2)},\R^{n(n+2)})$ at the origin. Consequently any holomorphic function on~$\tOmega$ that vanishes on $\tOmega\cap (\Lambda^n)^{n+2}$ must vanish in a neighborhood of the point~$(\be_0,\ldots,\be_0)$ in~$\tOmega$. Hence $H(\ba_0,\ldots,\ba_{n+1})=0$ for all $(\ba_0,\ldots,\ba_{n+1})\in\tOmega$ sufficiently close to~$(\be_0,\ldots,\be_0)$. 
But, by Lemma~\ref{lem_connect}, the manifold~$\tOmega$ is pathwise connected. Thus $H(\ba_0,\ldots,\ba_{n+1})=0$ for all $(\ba_0,\ldots,\ba_{n+1})\in\tOmega$.
\end{proof}

\section{Hereditary orderings on simplicial complexes}\label{section_ordering}

Recall that a simplicial complex~$K$ is called \textit{flag} if, for any vertices $v_0,\ldots,v_q$ of~$K$ such that $\{v_j,v_k\}\in K$ for all~$j$ and~$k$, the simplex $\{v_0,\ldots,v_q\}$ belongs to~$K$.

A \textit{facet\/} of a simplex is a codimension~$1$ face of it. A pair of non-empty simplices $(\sigma,\tau)$ of a simplicial complex~$K$ is called a \textit{free pair} if 
\begin{enumerate}
\item $\tau$ is a facet of~$\sigma$,
\item $\sigma$ is a maximal with respect to inclusion simplex of~$K$,
\item $\tau$ is not contained in any other simplex of~$K$, except for~$\sigma$.
\end{enumerate}
Equivalently, $(\sigma,\tau)$ is a free pair of~$K$ if $\tau$ is a facet of~$\sigma$ and the set $K_1=K\setminus\{\sigma,\tau\}$ is a simplicial complex. If this is the case, then we shall say that $K_1$ is obtained from~$K$ by the \textit{elementary simplicial collapse\/} of the free pair~$(\sigma,\tau)$. We shall say that a finite simplicial complex~$K$ \textit{collapses\/} on a subcomplex $L\subseteq K$ if $L$ can be obtained from~$K$ by a sequence of elementary simplicial collapses. It is easy to see that if $K$ collapses on~$L$, then $L$ is a deformation retract of~$K$.

Let $K$ be a finite simplicial complex. Let $\succ$ be a strict total ordering of the set~$K$. For each non-empty simplex $\sigma$ of~$K$, we denote by $\mu(\sigma)$ the largest facet of~$\sigma$ with respect to~$\succ$. The ordering $\succ$ will be called \textit{hereditary\/} if 
\begin{enumerate}
\item $\sigma\succ\tau$ whenever $\dim\sigma>\dim\tau$,
\item $\sigma\succ\tau$ whenever $\mu(\sigma)\succ\mu(\tau)$.
\end{enumerate}

The main fact on hereditary orderings that will be needed to us is the following proposition. This proposition is implicitly contained in~\cite{Gai12}. However, for the convenience of the reader, we shall prove it here. 

%Let $(K,\succ)$ be a simplicial complex with a hereditary ordering. We shall say that the pair $(K,\succ)$ is \textit{$r$-united\/} if, for any two simplices $\sigma,\tau\in K$  such that $\dim\sigma=\dim\tau= r$ and $\mu(\sigma)=\mu(\tau)$, we have $\sigma\cup\tau\in K$.

\begin{propos}\label{propos_sigma_tau}
Let~$K$ be a flag simplicial complex, let $\succ$ be a hereditary ordering of~$K$, and let $r$ be a positive integer. Suppose that for any two simplices $\sigma,\tau\in K$  such that $\dim\sigma=\dim\tau= r$ and $\mu(\sigma)=\mu(\tau)$, we have $\sigma\cup\tau\in K$. Then the simplicial complex~$K$ collapses on a subcomplex of dimension less than~$r$.
\end{propos}

To prove this proposition we need to introduce some auxiliary notation and to prove several simple properties of hereditary orderings. Below $(K,\succ)$ is always a simplicial complex with a hereditary ordering.   

Suppose that $\sigma\in K$ and $j\le \dim\sigma$. Consider all $j$-dimensional faces of~$\sigma$, and denote by~$\mu_j(\sigma)$ the largest of them with respect to~$\succ$.
In particular, $\mu_{\dim\sigma-1}(\sigma)=\mu(\sigma)$. It follows immediately from the definition of a hereditary ordering that 
$$
\emptyset=\mu_{-1}(\sigma)\subset\mu_0(\sigma)\subset\cdots\subset\mu_{\dim\sigma}(\sigma)=\sigma,
$$
and $\mu_j(\mu_k(\sigma))=\mu_j(\sigma)$ whenever $j\le k\le\dim\sigma$. Therefore, if $\dim\sigma=\dim\tau$ and there is a $j<\dim\sigma$ such that $\mu_j(\sigma)\succ\mu_j(\tau)$, then $\sigma\succ\tau$.

\begin{lem}\label{lem_sigma_cup_tau}
Suppose that $\sigma$ and~$\tau$ are simplices of~$K$ such that $\mu(\sigma)=\mu(\tau)$, $\sigma\succ\tau$, and $\sigma\cup\tau\in K$. Then $\mu(\sigma\cup\tau)=\sigma$.
\end{lem}

\begin{proof} Suppose that $\dim\sigma=k$. Since $\mu(\sigma)=\mu(\tau)$, we have $\dim\tau=k$ and $\dim(\sigma\cup\tau)=k+1$. For every $j=-1,\ldots,k$, we put~$\rho_j=\mu_j(\sigma\cup\tau)$. %In particular, $\rho_k=\mu(\sigma\cup\tau)$. Then $\emptyset=\rho_{-1}\subset\rho_0\subset\cdots\subset\rho_k$.

Let us prove by induction on~$j$ that~$\rho_j$ is contained in~$\sigma\cap\tau$ whenever $j\le k-1$. Since $\rho_{-1}=\emptyset$, the base of induction is trivial. Assume that $\rho_{j-1}\subseteq\sigma\cap\tau$ and prove that $\rho_{j}\subseteq\sigma\cap\tau$,  $j\le k-1$. The facet~$\rho_{j-1}$ of the $j$-dimensional simplex~$\rho_j$ is contained in~$\sigma\cap\tau$, and the simplex~$\rho_j$ itself is contained in~$\sigma\cup\tau$. It follows easily that $\rho_j$ is contained in at least one of the two simplices~$\sigma$ and~$\tau$. If $\rho_j$ were not contained in~$\sigma$, then it would be contained in~$\tau$, and we would obtain that $\mu_j(\tau)=\rho_j\succ\mu_j(\sigma)$, which is impossible since $\dim\sigma=\dim\tau$ and~$\sigma\succ\tau$. Hence $\rho_j$ is contained in~$\sigma$. Therefore $\rho_j=\mu_j(\sigma)$. Since $j\le k-1$, we obtain that $\rho_j\subseteq\mu(\sigma)=\sigma\cap\tau$.

Thus $\rho_{k-1}=\sigma\cap\tau$. Hence $\mu(\sigma\cup\tau)=\rho_k$ is either~$\sigma$ or~$\tau$. Since $\sigma\succ\tau$, we obtain that $\mu(\sigma\cup\tau)=\sigma$. 
\end{proof}

We denote by $M(K)$ the set of all simplices~$\sigma\in K$ such that $\sigma=\mu(\eta)$ for some $\eta\in K$.
Let $\sigma\in K$ be a simplex that is not maximal with respect to inclusion, and let $\dim\sigma=k$. Consider all $(k+1)$-dimensional simplices of~$K$ containing~$\sigma$, and denote by~$\lambda(\sigma)$ the smallest of them with respect to~$\succ$.

\begin{lem}\label{lem_lambda}
If $\sigma\in M(K)$, then $\mu(\lambda(\sigma))=\sigma$ and $\lambda(\sigma)\notin M(K)$.
\end{lem}

\begin{proof}
Assume that $\mu(\lambda(\sigma))=\tau\ne\sigma$. Then $\tau\succ\sigma$. Since $\sigma\in M(K)$, we obtain that there exists a simplex $\eta\in K$ such that $\mu(\eta)=\sigma$. Since $\eta\supset\sigma$ and $\dim\eta=\dim\sigma+1$, we see that $\eta\succ\lambda(\sigma)$. On the other hand, since $\mu(\lambda(\sigma))=\tau\succ\sigma=\mu(\eta)$, we obtain that $\lambda(\sigma)\succ\eta$. This contradiction proves that $\mu(\lambda(\sigma))=\sigma$.

Now, assume that $\lambda(\sigma)\in M(K)$, i.\,e., $\lambda(\sigma)=\mu(\xi)$ for some $\xi\in K$. We have $\dim\xi=\dim\sigma+2$. Let $\zeta$ be the facet of~$\xi$ such that $\zeta\supset\sigma$ and $\zeta\ne\lambda(\sigma)$. Since $\zeta\supset\sigma$ and $\dim\zeta=\dim\sigma+1$, we see that $\zeta\succ\lambda(\sigma)$. On the other hand, since $\lambda(\sigma)=\mu(\xi)$, we have $\lambda(\sigma)\succ\zeta$. This contradiction proves that $\lambda(\sigma)\notin M(K)$.
\end{proof}

\begin{lem}\label{lem_flag}
Let $(K,\succ)$ be a flag simplicial complex with a hereditary ordering, and let~$r$ be a positive integer. Suppose that for any two simplices $\sigma,\tau\in K$  such that $\dim\sigma=\dim\tau=r$ and $\mu(\sigma)=\mu(\tau)$, we have $\sigma\cup\tau\in K$. Then for any two simplices $\sigma,\tau\in K$  such that $\dim\sigma=\dim\tau>r$ and $\mu(\sigma)=\mu(\tau)$, we also have $\sigma\cup\tau\in K$.
\end{lem}

\begin{proof}
Assume that $\dim\sigma=\dim\tau>r$, $\mu(\sigma)=\mu(\tau)$, and $\sigma\ne\tau$. Let $u$ and $v$ be the vertices of the simplices~$\sigma$ and~$\tau$, respectively, opposite to their common facet \mbox{$\sigma\cap\tau=\mu(\sigma)=\mu(\tau)$}. We put, $\rho=\mu_{r-1}(\sigma\cap\tau)$, $\sigma'=\rho\cup\{u\}$, and $\tau'=\rho\cup\{v\}$. Since $\sigma\cap\tau=\mu(\sigma)=\mu(\tau)$, we obtain that $\rho=\mu_{r-1}(\sigma)=\mu_{r-1}(\tau)$. Hence $\rho=\mu(\sigma')=\mu(\tau')$. Since $\dim\sigma'=\dim\tau'=r$, it follows that $\sigma'\cup\tau'\in K$. Therefore $\{u,v\}\in K$. Since $K$ is a flag complex, we obtain that $\sigma\cup\tau\in K$.
\end{proof}

\begin{lem}\label{lem_sigma_tau}
Let~$K$ be a simplicial complex, let $\succ$ be a hereditary ordering of~$K$, and let~$r$ be a positive integer. Suppose that for any two simplices $\sigma,\tau\in K$  such that $\dim\sigma=\dim\tau\ge r$ and $\mu(\sigma)=\mu(\tau)$, we have $\sigma\cup\tau\in K$. Then the simplicial complex~$K$ collapses on a subcomplex of dimension less than~$r$.
\end{lem}

\begin{proof}
Let $\sigma_1,\ldots,\sigma_q$ 
be all simplices in~$K\setminus M(K)$ of dimensions greater than or equal to~$r$ ordered so that
$$
\sigma_1\succ\cdots\succ\sigma_q\,.
$$
For $j=0,\ldots,q$, we put
$$
K_j=K\setminus\{\sigma_1,\mu(\sigma_1),\sigma_2,\mu(\sigma_2),\ldots,\sigma_{j},\mu(\sigma_j)\}.
$$
Let us prove  that $K_j$ is a simplicial complex. Assume the converse. Then there exist simplices $\rho$ and $\tau$ such that $\rho$ is a facet of~$\tau$, $\tau\in K_j$, and $\rho\notin K_j$. Since $K$ is a simplicial complex, we have $\rho\in K$. Hence either $\rho=\sigma_i$ or $\rho=\mu(\sigma_i)$, where $1\le i\le j$. 

Assume that $\tau\in M(K)$. By Lemma~\ref{lem_lambda}, $\mu(\lambda(\tau))=\tau$ and $\lambda(\tau)\notin M(K)$. Besides, $\lambda(\tau)\succ\sigma_i$, since $$\dim\lambda(\tau)=\dim\rho+2> \dim\sigma_i.$$ Hence
 $\lambda(\tau)=\sigma_k$ for some $k<i$. Therefore $\tau=\mu(\sigma_k)$. Consequently $\tau\notin K_j$, which yields a contradiction. Thus $\tau\notin M(K)$. Since $\tau\in K_j$, it follows that $\sigma_j\succ\tau$, hence, $\sigma_i\succ\tau$. Since $\dim\tau=\dim\rho+1$, this would be impossible if $\rho$ were $\sigma_i$. Therefore $\rho=\mu(\sigma_i)$. 
 
Since $\rho$ is a facet of~$\tau$, $\rho=\mu(\sigma_i)$, and $\sigma_i\succ\tau$, we obtain that $\rho=\mu(\tau)$. Indeed, if $\mu(\tau)$ did not coincide with~$\rho$, we would have $\mu(\tau)\succ\rho=\mu(\sigma_i)$, hence, $\tau\succ\sigma_i$. We have $\dim\sigma_i=\dim\tau\ge r$ and $\mu(\sigma_i)=\mu(\tau)$. By Lemma~\ref{lem_flag}, $\sigma_i\cup\tau\in K$. By Lemma~\ref{lem_sigma_cup_tau}, we obtain that $\mu(\sigma_i\cup\tau)=\sigma_i$, which is impossible, since $\sigma_i\notin M(K)$. This contradiction proves that $K_j$ is a simplicial complex.

For every $j=0,\ldots,q-1$, we have $K_j=K_{j+1}\cup\{\sigma_j,\mu(\sigma_j)\}$, and both~$K_j$ and~$K_{j+1}$ are simplicial complexes. It follows easily that $(\sigma_j,\mu(\sigma_j))$ is a free pair in~$K_j$. Hence $K_{j+1}$ is obtained from~$K_j$ by the elementary collapse of this free pair. Thus $K=K_0$ collapses on~$K_q$.

Now, let us prove that $\dim K_q<r$. Assume that $K_q$ contains a simplex~$\tau$ of dimension greater than or equal to~$r$. Since all simplices $\sigma_1,\ldots,\sigma_q$ do not belong to~$K_q$, we obtain that $\tau\in M(K)$. By Lemma~\ref{lem_lambda}, $\lambda(\tau)\notin M(K)$ and $\mu(\lambda(\tau))=\tau$. Since $\dim\lambda(\tau)>r$, we have $\lambda(\tau)=\sigma_j$ for some~$j$. Then the simplex $\tau=\mu(\sigma_j)$ does not belong to~$K_q$, which yields a contradiction. Hence $\dim K_q<r$.
\end{proof}

Proposition~\ref{propos_sigma_tau} follows immediately from Lemmas~\ref{lem_flag} and~\ref{lem_sigma_tau}.

\section{Simplicial complexes $\CK(G,\varkappa)$}\label{section_KGkappa}

Let $G=(g_{uv})_{1\le u,v\le m}$ be a complex matrix of size $m\times m$ with units on the diagonal. Let $\varkappa$ be a positive number. Consider the graph $\Gamma(G,\varkappa)$ on the vertex set~$[m]=\{1,\ldots,m\}$ such that $\{u,v\}$ is an edge of~$\Gamma(G,\varkappa)$ if and only if $|g_{uv}-1|<\varkappa$. Let $\CK(G,\varkappa)$ be the \textit{clique complex\/} of~$\Gamma(G,\varkappa)$. This means that  $\CK(G,\varkappa)$ is the simplicial complex on the vertex set~$[m]$ such that a subset $\sigma\subseteq[m]$ belongs to~$\CK(G,\varkappa)$ if and only if the elements of~$\sigma$ are pairwise joined by edges of~$\Gamma(G,\varkappa)$.

It is not hard to see that the simplicial complexes~$\CK(G,\varkappa)$ have very specific properties if $G$ is the Gram matrix of points $\ba_1,\ldots,\ba_m\in\X^n$, where $\X^n$ is either~$\bS^n$ or~$\Lambda^n$. 
Indeed, if $\X^n=\bS^n$, then $g_{uv}=\cos\dist_{\bS^n}(\ba_u,\ba_v)$. Assume that the number $\varkappa$ is chosen so that, for any~$u$ and~$v$, we have either $|g_{uv}-1|>\varkappa$ or $|g_{uv}-1|<\varkappa/4$. Then it follows easily from the triangle inequality and the inequality
$$
1-\cos 2x\le 4(1-\cos x)
$$
that $\{u,v\}$ is an edge of $\Gamma(G,\varkappa)$ whenever there is a~$w$ such that $\{u,w\}$ and $\{v,w\}$ are edges of~$ \Gamma(G,\varkappa)$. Similarly, if $\X^n=\Lambda^n$, then $g_{uv}=\cosh\dist_{\Lambda^n}(\ba_u,\ba_v)$. Assume that the number $\varkappa$ is chosen so that $\varkappa<1$ and for any~$u$ and~$v$, we have either $|g_{uv}-1|>\varkappa$ or $|g_{uv}-1|<\varkappa/6$. It is easy to see that 
$$
\cosh 2x-1\le 6(\cosh x-1)
$$
whenever $\cosh x-1\le 1$. This inequality and the triangle inequality in~$\Lambda^n$ again imply
that $\{u,v\}$ is an edge of $\Gamma(G,\varkappa)$ whenever there is a~$w$ such that $\{u,w\}$ and $\{v,w\}$ are edges of~$ \Gamma(G,\varkappa)$. Hence, in both cases, $\Gamma(G,\varkappa)$ is the disjoint union of complete graphs. Therefore $\CK(G,\varkappa)$ is the disjoint union of simplices.

If $G$ is the Gram matrix of arbitrary vectors $\ba_1,\ldots,\ba_m\in Q$, then no analog of the triangle inequality holds. Actually, for~$G$ we can obtain any symmetric matrix with units on the diagonal of rank less than or equal to~$n+1$. Nevertheless, it turns out that the bound on the rank of~$G$ still implies strong restrictions on the topology of the complex~$\CK(G,\varkappa)$, provided that $\varkappa<1$ and for any $u,v\in[m]$, we have either $|g_{uv}-1|>\varkappa$ or $|g_{uv}-1|<\varkappa/c$ for some sufficiently large constant $c>0$.

\begin{theorem}\label{theorem_KCkappa}
Let $G=(g_{uv})_{1\le u,v\le m}$ be a complex matrix of size $m\times m$ with units on the diagonal such that $\rank G\le 2r$. Suppose that $0<\varkappa<1,$ and for any $u,v\in[m],$ we have either $|g_{uv}-1|>\varkappa$ or $ |g_{uv}-1|<2^{-4(r+1)}\varkappa$. Then the simplicial complex~$\CK(G,\varkappa)$ collapses on a subcomplex of dimension less than~$r$.
\end{theorem}

\begin{cor}\label{cor_KCkappa}
Let $G=(g_{uv})_{1\le u,v\le m}$ be a complex matrix of size $m\times m$ with units on the diagonal such that $\rank G\le 2r$. Then there exists a $\varkappa$ such that 
$$
2^{-2m^2(r+1)}<\varkappa<1
$$
and the simplicial complex~$\CK(G,\varkappa)$ collapses on a subcomplex of dimension less than~$r$.
\end{cor}

\begin{proof}
Consider the interval $I=\bigl(-(2m^2+4)(r+1),0\bigr)$ and consider the $m(m-1)/2$ points $\log_{2}|g_{uv}-1|\in\R\cup\{-\infty\}$, $u<v$. Obviously,  $I$ contains a closed interval~$[\alpha,\beta]$ of length~$4(r+1)$ such that $[\alpha,\beta]$ contains none of the points~$\log_{2}|g_{uv}-1|$. We put $\varkappa=2^{\beta}$. Then $2^{-2m^2(r+1)}<\varkappa<1$ and none of the numbers $|g_{uv}-1|$ belongs to~$[2^{-4(r+1)}\varkappa,\varkappa]$. By Theorem~\ref{theorem_KCkappa}, $\CK(G,\varkappa)$ collapses on a subcomplex of dimension less than~$r$.
\end{proof}

\begin{remark}
Theorem~\ref{theorem_KCkappa}  can be regarded as a `quantitative' (and, at the same time, non-Euclidean) version of Main Lemma in~\cite[Sect.~5]{Gai12}. This Main Lemma concerns \textit{places\/} of the field of rational functions in the vertex coordinates of the polyhedron. To each such place~$\varphi$ corresponds the graph~$\Gamma_{\varphi}$ such that $\{u,v\}$ is an edge of~$\Gamma_{\varphi}$ if and only if the length~$\ell_{uv}$ is finite at~$\varphi$. Then Main Lemma claims that the clique complex~$\CK_{\varphi}$ of~$\Gamma_{\varphi}$ collapses on a subcomplex of dimension less than or equal to~$[n/2]$, where $n$ is the dimension of the polyhedron. In Theorem~\ref{theorem_KCkappa},  we in a sense replace the word `finite' with the words `sufficiently small' and the word `infinite' with the words `sufficiently large'.
\end{remark}

In the rest of this section, we shall prove Theorem~\ref{theorem_KCkappa}.

%Let $(G,\varkappa)$ be an arbitrary pair not necessarily satisfying the hypothesis of Theorem~\ref{theorem_KCkappa}.
Let us introduce a special hereditary ordering~$\succ$ on~$\CK=\CK(G,\varkappa)$. Since we want the ordering~$\succ$ to be hereditary, we must put $\sigma\succ\tau$ whenever $\dim\sigma>\dim\tau$. So we need to order $s$-dimensional simplices for every~$s$. We shall do it consecutively for $s=0,\ldots,\dim \CK$. 

First, we choose an arbitrary total ordering~$\succ$ on the set of vertices of~$\CK$.

Further, we proceed by induction on dimension. 
Suppose that we have already determined the ordering~$\succ$ on the set of $(s-1)$-dimensional simplices of~$\CK$. Construct the ordering~$\succ$ on the set of $s$-dimensional simplices of~$\CK$. Since we want the ordering~$\succ$ to be hereditary, we must put $\sigma\succ\tau$ whenever $\dim\sigma=\dim\tau=s$ and $\mu(\sigma)\succ\mu(\tau)$. 

Now, for each $(s-1)$-simplex $\rho\in \CK$, we need to order $s$-dimensional simplices $\sigma\in \CK$ with $\mu(\sigma)=\rho$. Each such simplex $\sigma$ has the form $\rho\cup\{v\}$. Denote by~$\V_{\rho}$ the set of all vertices $v\notin\rho$ such that $\rho\cup\{v\}\in\CK$ and $\mu(\rho\cup\{v\})=\rho$. Let~$t$ be the cardinality of~$\V_{\rho}$. 
We shall successively denote the elements of~$\V_{\rho}$ by $v_1,\ldots,v_t$ in the following way. Suppose that the vertices $v_1,\ldots,v_{j-1}$ have already been chosen, and $j<t$. 
Consider all values $|g_{wz}-1|$, where $w,z\in \V_{\rho}\setminus\{v_1,\ldots,v_{j-1}\}$ and $w\ne z$, and choose the maximum of them. Let the maximum be attained at a pair $(w^0,z^0)$. (If the maximum is attained at several pairs, then we take for $(w^0,z^0)$ any of them.) Then we put $v_j=w^0$. Finally, the unique vertex in~$\V_{\rho}\setminus\{v_1,\ldots,v_{t-1}\}$ is taken for $v_t$. Now, $\rho\cup\{v_j\}$, $j=1,\ldots,t$, are exactly all simplices in $\CK$  with maximal facet~$\rho$. We put
\begin{equation}\label{eq_rho_vj}
\rho\cup\{v_1\}\succ\rho\cup\{v_2\}\succ\cdots\succ\rho\cup\{v_t\}.
\end{equation}
This completes the construction of the ordering~$\succ$.

For a vertex $v\in \V_{\rho}$, we denote by $\V_{\rho}(v)$ the subset of~$\V_{\rho}$ consisting of all vertices~$u$ such that $\sigma\cup\{v\}\succ\sigma\cup\{u\}$. 

\begin{lem}\label{lem_ii}
Suppose that $\rho\in \CK,$ $\rho\ne\emptyset,$ $v\in \V_{\rho},$ and\/ $\V_{\rho}(v)\ne \emptyset$.  Then there is a vertex $u\in \V_{\rho}(v)$ such that
$$
|g_{wz}-1|\le |g_{uv}-1|
$$
for all $w,z\in \V_{\rho}(v)\cup\{v\}$.
\end{lem}

\begin{proof}
As above, let $v_1,\ldots,v_t$ be all elements of the set $\V_{\rho}$ numbered so that \eqref{eq_rho_vj} holds true. Assume that $v=v_j$. Then $\V_{\rho}(v)=\{v_{j+1},\ldots,v_t\}$. Since $\V_{\rho}(v)\ne\emptyset$, we have $j<t$. By the construction of the ordering~$\succ$, the maximum of all values $|g_{wz}-1|$ such that $w,z\in \V_{\rho}(v)\cup\{v\}=\{v_{j},\ldots,v_t\}$ and $w\ne z$ is attained at a pair $(w^0,z^0)$ such that $w^0=v$. Then $u=z^0$ is the required vertex.
\end{proof}

\begin{lem}\label{lem_tech_nondeg}
Let $B=(b_{jk})_{1\le j,k\le n}$ be a complex matrix with units on the diagonal such that $|b_{jk}|<3/2$ whenever $j<k$ and $|b_{jk}|<4^{-n}$ whenever $j>k$. Then the matrix~$B$ is non-degenerate.
\end{lem}

\begin{proof}
We shall prove the following stronger assertion:

\textit{Let $B=(b_{jk})_{1\le j,k\le n}$ be a complex matrix such that $1-4^{-n}<|b_{jj}|<3/2+4^{-n}$ for all~$j$, $|b_{jk}|<3/2+4^{-n}$ whenever $j<k,$ and $|b_{jk}|<4^{-n}$ whenever $j>k$. Then the matrix~$B$ is non-degenerate.}

Let us prove this assertion by induction on~$n$. If $n=1$, then it is obviously true. Assume that this assertion is true for $n-1$ and prove it for $n$.

For every $j=2,\ldots,n$, we subtract from the $j$th row of~$B$ the first row of~$B$ multiplied by $b_{j1}/b_{11}$. Let $B'=(b'_{jk})$ be the matrix obtained. Then all entries of the first column of~$B'$, except for $b_{11}'=b_{11}$, are equal to zero. If $2\le j,k\le n$, we have $b_{jk}'=b_{jk}-b_{1k}b_{j1}/b_{11}$. Since $|b_{11}|>2/3$, $|b_{1k}|<2$, and $|b_{j1}|<4^{-n}$, we obtain that $|b_{jk}'-b_{jk}|<3\cdot 4^{-n}$. Hence $1-4^{-n+1}<|b_{jj}'|<3/2+4^{-n+1}$ for all~$j$, $|b'_{jk}|<3/2+4^{-n+1}$ whenever $j<k$ and $|b'_{jk}|<4^{-n+1}$ whenever $j>k$. Therefore, by the inductive assumption, $B'$ is non-degenerate. Thus $B$ is also non-degenerate.
\end{proof}

\begin{lem}\label{lem_C_sigma_tau}
Let the triple $(G,\varkappa,r)$ be as in Theorem~\ref{theorem_KCkappa}, let $\CK=\CK(G,\varkappa)$, and let~$\succ$ be the hereditary ordering on~$\CK$ constructed above.
%Let $r$ be a positive integer. Suppose that\/ $\rank G\le 2r,$ $0<\varkappa<1,$ and for any $u,v\in[m],$ we have either $|g_{uv}-1|>\varkappa$ or $ |g_{uv}-1|<2^{-4(r+1)}\varkappa$. 
Then for any two simplices $\sigma,\tau\in \CK$  such that $\dim\sigma=\dim\tau= r$ and $\mu(\sigma)=\mu(\tau),$ we have $\sigma\cup\tau\in \CK$. 
\end{lem}

\begin{proof}
We may assume that $\sigma\succ\tau$. For $j=0,\ldots,r$, we put $\rho_j=\mu_j(\sigma)$. (Recall that~$\mu_j(\sigma)$ is the largest with respect to~$\succ$ of all $j$-dimensional faces of~$\sigma$.) Then 
\begin{gather*}
\rho_0\subset\cdots\subset\rho_{r-1}\subset\rho_r=\sigma,\\
\rho_{r-1}=\mu(\sigma)=\mu(\tau)=\sigma\cap\tau.
\end{gather*}
%We put $\sigma_{r}=\sigma$, and successively put $\sigma_j=\mu(\sigma_{j+1})$ for $j=r-1,\ldots,1,0$. Since the ordering~$\succ$ is hereditary, we easily see that $\sigma_j$ is the maximal $j$-dimensional face of~$\sigma$ with respect to~$\succ$.
We denote the vertices of~$\sigma$ by~$v_0,\ldots,v_r$ so that $\rho_j=\{v_0,\ldots,v_j\}$, $j=0,\ldots,r$. We denote by $u_r$ the vertex of~$\tau$ opposite to the facet $\mu(\tau)=\rho_{r-1}$; then $\tau=\{v_0,\ldots,v_{r-1},u_r\}$. Since $\sigma\succ\tau$, we have $u_k\in \V_{\rho_{k-1}}(v_k)$.

We need to prove that $|g_{u_rv_r}-1|<\varkappa$. If we prove this, we shall obtain that $\{u_r,v_r\}\in \CK$, hence, $\sigma\cup\tau\in \CK$.  Assume the converse, i.\,e., assume that $|g_{u_rv_r}-1|\ge\varkappa$.

For every $j=1,\ldots,r-1$, we have $v_j\in\V_{\rho_{j-1}}$, since $\mu(\rho_j)=\rho_{j-1}$. Besides, the set $\V_{\rho_{j-1}}(v_j)$ is non-empty, since it contains the vertex~$v_{j+1}$. By Lemma~\ref{lem_ii},  there is a vertex~$u_j\in \V_{\rho_{j-1}}(v_j)$ such that 
\begin{equation}\label{eq_finite_odd}
|g_{wz}-1|\le|g_{u_jv_j}-1|
\end{equation}
for all $w,z\in \V_{\rho_{j-1}}(v_j)\cup\{v_j\}$.
 
We have, 
\begin{equation}\label{eq_determinant}
\left|\,\,\,\,
\begin{matrix}
1&g_{v_0v_1}&g_{v_0v_2}&\cdots&g_{v_0v_r}&\vline & g_{v_0u_1} & g_{v_0u_2}  &\cdots &g_{v_0u_r}\\
g_{v_1v_0}&1&g_{v_1v_2}&\cdots&g_{v_1v_r}& \vline &g_{v_1u_1} & g_{v_1u_2}  &\cdots &g_{v_1u_r}\\
g_{v_2v_0}&g_{v_2v_1}&1&\cdots&g_{v_2v_r}& \vline &g_{v_2u_1} & g_{v_2u_2}  &\cdots &g_{v_2u_r}\\
\vdots &\vdots  & \vdots &\ddots& \vdots& \vline &\vdots &\vdots & \ddots &\vdots\\
g_{v_rv_0}&g_{v_rv_1}&g_{v_rv_2}&\cdots&1& \vline &g_{v_ru_1} & g_{v_ru_2}  &\cdots &g_{v_ru_r}\\
\hline
g_{u_1v_0}&g_{u_1v_1}&g_{u_1v_2}&\cdots& g_{u_1v_r}&\vline  &  1 &g_{u_1u_2}  &  \cdots & g_{u_1u_r}\\
g_{u_2v_0}&g_{u_2v_1}&g_{u_2v_2}&\cdots& g_{u_2v_r}&\vline   &g_{u_2u_1}  &  1 &  \cdots & g_{u_2u_r}\\
\vdots & \vdots & \vdots &\ddots &\vdots & \vline& \vdots &\vdots & \ddots &\vdots\\
g_{u_rv_0}&g_{u_rv_1}&g_{u_rv_2}&\cdots& g_{u_rv_r}&\vline  &g_{u_ru_1}&g_{u_ru_2}  &  \cdots & 1\\
\end{matrix}
\,\,\,\,\right|=0.
\end{equation}
Indeed, if the $2r+1$ vertices $v_0,\ldots,v_r,u_1,\ldots,u_r$ are pairwise distinct, then the determinant in the left-hand side of~\eqref{eq_determinant} is a minor of~$G$ of order $2r+1$. It vanishes, since $\rank G\le 2r$. If two of the vertices $v_0,\ldots,v_r,u_1,\ldots,u_r$ coincide to each other, then the determinant in the left-hand side of~\eqref{eq_determinant} contains two identical rows, hence, vanishes.
 
 For any $w,z\in [m]$, we put $q_{wz}=g_{wz}-1$. Then
 \begin{equation}\label{eq_ineq_q}
 \begin{aligned}
 |q_{wz}|&<2^{-4(r+1)}\varkappa&&\text{if }\{w,z\}\in \CK,\\
 |q_{wz}|&>\varkappa&&\text{if }\{w,z\}\notin \CK.
 \end{aligned}
 \end{equation}
 
Performing the first step of the Gaussian elimination for the determinant in the left-hand side of~\eqref{eq_determinant}, we obtain that
\begin{equation}\label{eq_determinant2}
\left|\,\,\,\,
\begin{matrix}
a_{v_1v_1}&a_{v_1v_2}&\cdots&a_{v_1v_r}& \vline &a_{v_1u_1} & a_{v_1u_2}  &\cdots &a_{v_1u_r}\\
a_{v_2v_1}&a_{v_2v_2}&\cdots&a_{v_2v_r}& \vline &a_{v_2u_1} & a_{v_2u_2}  &\cdots &a_{v_2u_r}\\
\vdots  & \vdots &\ddots& \vdots& \vline &\vdots &\vdots & \ddots &\vdots\\
a_{v_rv_1}&a_{v_rv_2}&\cdots&a_{v_rv_r}& \vline &a_{v_ru_1} & a_{v_ru_2}  &\cdots &a_{v_ru_r}\\
\hline
a_{u_1v_1}&a_{u_1v_2}&\cdots& a_{u_1v_r}&\vline  &  a_{u_1u_1} &a_{u_1u_2}  &  \cdots & a_{u_1u_r}\\
a_{u_2v_1}&a_{u_2v_2}&\cdots& a_{u_2v_r}&\vline   &a_{u_2u_1}  &  a_{u_2u_2} &  \cdots & a_{u_2u_r}\\
\vdots & \ddots &\vdots &\vdots & \vline& \vdots &\vdots & \ddots &\vdots\\
a_{u_rv_1}&a_{u_rv_2}&\cdots& a_{u_rv_r}&\vline  &a_{u_ru_1}&a_{u_ru_2}  &  \cdots & a_{u_ru_r}\\
\end{matrix}
\,\,\,\,\right|=0,
\end{equation}
where 
$$a_{wz}=a_{zw}=g_{wz}-g_{v_0w}g_{v_0z}=
q_{wz}-q_{v_0w}-q_{v_0z}-q_{v_0w}q_{v_0z}.$$
Since $\varkappa<1$,  $r\ge 1$, and $\{v_0,w\}\in \CK$ whenever $w\in\{v_1,\ldots,v_r,u_1,\ldots,u_r\}$, inequalities~\eqref{eq_ineq_q} imply that, for $z,w\in\{v_1,\ldots,v_r,u_1,\ldots,u_r\}$, we have
\begin{gather}\label{eq_ineq_a1}
 |a_{wz}|<2^{-4r-2}\varkappa\qquad\text{if }\{w,z\}\in \CK,\\
 \label{eq_ineq_a2}\frac{99}{100}\,|q_{wz}| < |a_{wz}|<\frac{101}{100}\,|q_{wz}|\qquad\text{if }\{w,z\}\notin \CK.
 \end{gather}  
 
Let $1\le j<k\le r$, and  let $w$ be either $u_k$ or~$v_k$. Since $w\in  \V_{\rho_{k-1}}$, we see that
$\rho_{l}\cup\{w\}\in \CK$ and $\mu(\rho_{l}\cup\{w\})=\rho_{l}$ for  $l=0,\ldots,k-1$. In particular, $\mu(\rho_{j-1}\cup\{w\})=\rho_{j-1}$ and $\mu(\rho_j\cup\{w\})=\rho_j$. Hence $w\in \V_{\rho_{j-1}}$ and $\rho_j\succ\rho_{j-1}\cup\{w\}$. Therefore $w\in \V_{\rho_{j-1}}(v_j)$. Thus $u_k,v_k\in  \V_{\rho_{j-1}}(v_j)$ whenever $1\le j<k\le r$. Hence inequality~\eqref{eq_finite_odd} implies that 
\begin{equation}\label{eq_finite_concrete}
|q_{u_kv_k}|,\,|q_{u_jv_k}|,\,|q_{u_ju_k}|\le |q_{u_jv_j}|,\qquad 1\le j<k\le r.
\end{equation}
In particular, for $j=1,\ldots,r$, we have 
$$
|q_{u_jv_j}|\ge |q_{u_rv_r}|\ge\varkappa,
$$
hence, $\{u_j,v_j\}\notin \CK$. Therefore,
\begin{equation}\label{eq_a_diag}
|a_{u_jv_j}|> \frac{99}{100}\,\varkappa, \qquad j=1,\ldots,r.
\end{equation}
Now, inequalities~\eqref{eq_ineq_a2} and~\eqref{eq_finite_concrete} imply that
\begin{equation}\label{eq_finite_concrete2}
|a_{u_kv_k}|,\,|a_{u_jv_k}|,\,|a_{u_ju_k}|< \frac{101}{99}\,|a_{u_jv_j}|,\qquad 1\le j<k\le r.
\end{equation}
Besides, by~\eqref{eq_ineq_a1}, we have 
\begin{equation}\label{eq_ineq_a_triv}
|a_{v_jv_k}|< 2^{-4r-2}\varkappa,\qquad |a_{v_ju_k}|< 2^{-4r-2}\varkappa,\qquad 1\le j<k\le r,
\end{equation}
since $\{v_j,v_k\}$ and $\{v_j,u_k\}$ are edges of~$\CK$.

Apply to the matrix in the left-hand side of~\eqref{eq_determinant2} the following elementary row and column operations:
\begin{enumerate}
\item Permute the rows by the permutation 
$$
\begin{pmatrix}
1&2&\cdots&r&r+1&r+2&\cdots &2r\\
2r&2r-1&\cdots&r+1&1&2&\cdots &r
\end{pmatrix}.
$$
(The first row of the initial matrix becomes the $(2r)$th row of the matrix obtained, the second row of the initial matrix becomes the $(2r-1)$st row of the matrix obtained, etc.)
\item Permute the columns by the permutation 
$$
\begin{pmatrix}
1&2&\cdots&r&r+1&r+2&\cdots &2r\\
1&2&\cdots&r&2r&2r-1&\cdots &r+1
\end{pmatrix}.
$$
\item For every $j=1,\ldots,r$, multiply the $j$th row of the obtained matrix by~$a_{u_rv_r}/a_{u_jv_j}$, divide the $j$th column of the obtained matrix by $a_{u_rv_r}$, and divide the $(r+j)$th column of the obtained matrix by~$a_{u_{r+1-j}v_{r+1-j}}$.
\end{enumerate}

Let $B=(b_{\alpha\beta})_{1\le \alpha,\beta\le 2r}$ be the obtained matrix. Then for $1\le j,k\le r$, we have
\begin{align*}
&b_{jk}=\frac{a_{u_jv_k}}{a_{u_jv_j}}\,,&
&b_{j,r+k}=\frac{a_{u_ju_{r+1-k}}a_{u_rv_r}}{a_{u_jv_j}a_{u_{r+1-k}v_{r+1-k}}}\,,\\
&b_{r+j,k}=\frac{a_{v_{r+1-j}v_k}}{a_{u_rv_r}}\,,&
&b_{r+j,r+k}=\frac{a_{v_{r+1-j}u_{r+1-k}}}{a_{u_{r+1-k}v_{r+1-k}}}\,.
\end{align*}
In particular, all diagonal entries of~$B$ are equal to~$1$. Inequalities~\eqref{eq_a_diag}--\eqref{eq_ineq_a_triv} easily imply that 
$$
\begin{aligned}
|b_{\alpha\beta}|&<\left(\frac{101}{99}\right)^2<\frac32&&\text{if}\ 1\le \alpha<\beta\le 2r,\\ 
|b_{\alpha\beta}|&<\frac{100}{99}\cdot2^{-4r-2}<4^{-2r}&&\text{if}\ 1\le \beta<\alpha\le 2r.
\end{aligned}
$$
The matrix~$B$ was obtained from a degenerate matrix by elementary row and column operations. Hence  $B$ is degenerate. On the other hand, 
by Lemma~\ref{lem_tech_nondeg}, $B$ is non-degenerate. This contradiction proves that $|g_{u_rv_r}-1|<\varkappa$. Hence $\{u_r,v_r\}\in \CK$. Therefore $\sigma\cup\tau\in \CK$.
\end{proof}

\begin{proof}[Proof of Theorem~\ref{theorem_KCkappa}]
By Lemma~\ref{lem_C_sigma_tau}, for any two simplices $\sigma,\tau\in \CK(G,\varkappa)$  such that $\dim\sigma=\dim\tau= r$ and $\mu(\sigma)=\mu(\tau)$, we have $\sigma\cup\tau\in \CK(G,\varkappa)$. The simplicial complex $\CK(G,\varkappa)$ is a flag simplicial complex, since it is the clique complex of the graph~$\Gamma(G,\varkappa)$. Hence, by Proposition~\ref{propos_sigma_tau}, the simplicial complex~$\CK(G,\varkappa)$ collapses on a subcomplex of dimension less than~$r$.
\end{proof}

\section{Proof of the main theorem}\label{section_proof}

Let $s$ be the number of edges of the simplicial complex $K=\supp\xi$. Recall that the configuration space $\Sigma=\Sigma(\xi,\X^n,\bell)\subset(\X^n)^m$ of all simplicial polyhedra $P\colon K\to\X^n$ of combinatorial type~$\xi$ with the prescribed set of edge lengths~$\bell$ is given by $s$ polynomial equations~\eqref{eq_polynom}.
Recall that $\bS^n=Q\cap \R^{n+1}$ and $\Lambda^n$ is one of the two sheets of the two-sheeted hyperboloid $Q\cap\R^{1,n}$. Hence it is natural to consider the \textit{complexification\/} of~$\Sigma$, i.\,e., the subset $\Sigma_{\C}\subset Q^m$ given by the same $s$ polynomial equations~\eqref{eq_polynom}. 

We shall always identify a point $(\ba_1,\ldots,\ba_m)\in\Sigma_{\C}$ with the matrix~$A$ of size $(n+1)\times m$ with the columns~$\ba_1,\ldots,\ba_m$.  Then the Gram matrix of the vectors~$\ba_1,\ldots,\ba_m$ is the matrix $G=A^tA$.

The subset $\Sigma_{\C}\subset\C^{m(n+1)}$ is a (not necessarily irreducible) affine algebraic variety. We would like to study certain holomorphic functions on (relatively) open subsets $U\subseteq\Sigma_{\C}$. However, the variety~$\Sigma_{\C}$ may be not smooth. So we need to define a holomorphic function on a not necessarily smooth affine algebraic variety.

A function~$\varphi(\bz)$ defined on a relatively open subset $U$ of an affine variety $X\subseteq\C^N$ is called \textit{holomorphic\/} if every point $\bz\in U$ has a neighborhood $W$ in~$\C^N$ such that the function $\psi(\bz)$ can be extended to a  holomorphic function in~$W$, cf.~\cite[Ch.~V, Sect.~B]{GuRo65}. If $U$ consists of regular points of~$X$, then $U$ is a complex analytic manifold, and the above definition of a holomorphic function is equivalent to the standard definition of a holomorphic function on a complex analytic manifold.  We shall need the following version of the  Casorati--Sokhotski--Weierstrass theorem. 

\begin{theorem}
Let $\varphi$ be a non-constant holomorphic function on an irreducible affine variety $X\subseteq\C^N$. Then the image $\varphi(X)$ is dense in~$\C$.
\end{theorem}

For a smooth~$X$, this theorem can be found in~\cite[Prop.~2.4 and Prop.~3.8]{GrKi73}. The case of an arbitrary~$X$ follows immediately, since any irreducible affine variety contains a smooth principal Zariski open subset~$Y$, which also has a structure of an irreducible affine variety. 

\begin{cor}\label{cor_bhf}
Any bounded holomorphic function on an irreducible affine variety is constant.
\end{cor}

\begin{propos}\label{propos_Phi}
Suppose that $n\ge 3$ and $\ell_{uv}<2^{-m^2(n+4)}$ for all edges~$\{u,v\}$ of\/~$K$. Then there exists a bounded holomorphic function\/~$\Phi(A)$ on\/~$\Sigma_{\C}$ such that the restriction of\/~$\Phi(A)$ to\/~$\Sigma$ coincides with the function\/~$\CV_{\xi}(A)$.
\end{propos}

\begin{proof}
Let $\CK_1,\ldots,\CK_L$ be all simplicial complexes on the vertex set~$[m]$ that satisfy the following two conditions:
{\renewcommand{\theenumi}{\roman{enumi}}
\begin{enumerate}
\item There exists a point $A\in\Sigma_{\C}$ and a $\varkappa\in(2^{-m^2(n+4)},1)$   such that $\CK_l=\CK(A^tA,\varkappa)$.
\item $\CK_l$ collapses on a subcomplex of dimension less than~$n-1$.
\end{enumerate}
}

For $l=1,\ldots,L$, we denote by~$U_l$ the subset of~$\Sigma_{\C}$ consisting of all points $A$ such that $\CK_l=\CK(A^tA,\varkappa)$ for some $\varkappa\in(2^{-m^2(n+4)},1)$. By condition~(i), $U_l\ne \emptyset$.

\begin{lem}\label{lem_KU}
\textnormal{(1)} $K\subseteq \CK_l,$ $l=1,\ldots,L$.

\textnormal{(2)} If\/ $U_{l_1}\cap U_{l_2}\ne\emptyset,$ then either\/ $\CK_{l_1}\subseteq\CK_{l_2}$ or\/ $\CK_{l_2}\subseteq\CK_{l_1}$.

\textnormal{(3)} $U_1,\ldots,U_L$ are open subsets of\/~$\Sigma_{\C}$.

\textnormal{(4)} $\bigcup_{l=1}^LU_l=\Sigma_{\C}$.
\end{lem}
\begin{proof}
(1) We have $\CK_l=\CK(A^tA,\varkappa)$ for certain $A\in\Sigma_{\C}$ and certain $\varkappa\in(2^{-m^2(n+4)},1)$.     We put, $A^tA=(g_{uv})$.
By~\eqref{eq_polynom}, for all $\{u,v\}\in K$, we have $g_{uv}=\cos\ell_{uv}$ if $\X^n=\bS^n$ and $g_{uv}=\cosh\ell_{uv}$ if $\X^n=\Lambda^n$. Since $0\le \ell_{uv}<1$, we obtain that 
 \begin{equation*}
 |g_{uv}-1|\le\ell_{uv}<2^{-m^2(n+4)}<\varkappa,\qquad \{u,v\}\in K.
 \end{equation*}
Hence all edges of~$K$ are edges of~$\CK(A^tA,\varkappa)$. Since $\CK(A^tA,\varkappa)$ is a flag complex, we see that $K\subseteq\CK(A^tA,\varkappa)$. 
 
 (2) Suppose that $A\in U_{l_1}\cap U_{l_2}$; then $\CK_{l_1}=\CK(A^tA,\varkappa_1)$ and $\CK_{l_2}=\CK(A^tA,\varkappa_2)$ for some $\varkappa_1$ and~$\varkappa_2$. If $\varkappa_1\le \varkappa_2$, then $\CK_{l_1}\subseteq\CK_{l_2}$, and if $\varkappa_2\le \varkappa_1$, then $\CK_{l_2}\subseteq\CK_{l_1}$.
 
(3) Suppose that $A\in U_l$; then $\CK_l=\CK(A^tA,\varkappa)$ for certain $\varkappa\in(2^{-m^2(n+4)},1)$. We put, $A^tA=(g_{uv})$. Then $|g_{jk}-1|<\varkappa$ whenever $\{u,v\}\in\CK_l$ and $|g_{jk}-1|\ge\varkappa$ whenever $\{u,v\}\notin\CK_l$. Hence there is a $\varkappa\,'\in(2^{-m^2(n+4)},\varkappa)$ such that $|g_{jk}-1|<\varkappa\,'$ whenever $\{u,v\}\in\CK_l$ and $|g_{jk}-1|>\varkappa\,'$ whenever $\{u,v\}\notin\CK_l$. Then $\CK_l=\CK(B^tB,\varkappa\,')$ for all points $B\in\Sigma_{\C}$ that are close enough to~$A$. Therefore $U_l$ contains a neighborhood of~$A$ in~$\Sigma_{\C}$. Thus $U_l$ is an open subset of~$\Sigma_{\C}$.
 
(4) Let $A$ be an arbitrary point in~$\Sigma_{\C}$.  Since $A$ is a matrix of size $(n+1)\times m$, the rank of~$A^tA$ does not exceed $n+1$. Put $r=[n/2]+1$; then $\rank (A^tA)\le 2r$. By Corollary~\ref{cor_KCkappa}, there exists a $\varkappa\in(2^{-m^2(n+4)},1)$ such that the simplicial complex~$\CK(A^tA,\varkappa)$ collapses on a subcomplex of dimension less than~$r$. Since $n\ge 3$, we have $r\le n-1$. Hence the simplicial complex~$\CK(A^tA,\varkappa)$ satisfies conditions~(i) and~(ii). Therefore $\CK(A^tA,\varkappa)=\CK_l$ for certain~$l$. Then $A\in U_l$.
\end{proof}

Suppose that $A=(\ba_1,\ldots,\ba_m)\in U_l$. Since $\CK_l=\CK(A^tA,\varkappa)$ for certain $\varkappa<1$, we obtain that $|\langle\ba_{u},\ba_{v}\rangle-1|<1$ for all edges $\{u,v\}\in\CK_l$. Hence $(\ba_{u_0},\ldots,\ba_{u_n})\in\Omega$ for all $\{u_0,\ldots,u_n\}\in\CK_l$. Recall that the function~$F(A)$ given by~\eqref{eq_F(A)} is well defined and holomorphic on~$\Omega$. Therefore the formula 
\begin{equation}\label{eq_fl}
f_l(A)([u_0,\ldots,u_n])=F(\ba_{u_0},\ldots,\ba_{u_n})
\end{equation}
determines a well-defined holomorphic mapping
$$
f_l\colon U_l\to C^n(\CK_l;\C).
$$
It follows from Lemma~\ref{lem_zero_sum} that $\delta f_l(A)=0$ in $C^{n+1}(\CK_l;\C)$, i.\,e., $f_l(A)\in Z^{n}(\CK_l;\C)$ for all $A\in U_l$.

%Let us proceed with the proof of Proposition~\ref{propos_Phi}. 
By Lemma~\ref{lem_KU}, the complex $K=\supp\xi$ is contained in~$\CK_l$ for all~$l$. Hence,  for each~$l$, we can consider~$\xi$ as an element of~$C_{n-1}(\CK_l)$. Since $\CK_l$ collapses on a subcomplex of dimension less than~$n-1$, we obtain that the homology groups~$H_j(\CK_l)$ are trivial for all $j\ge n-1$.  Hence there exists a chain $\eta_l\in C_{n}(\CK_l)$ such that $\partial\eta_l=\xi$.  We put
\begin{equation}\label{eq_Phi_l}
\Phi_l(A)=\frac{2}{\nu^n\,\Gamma\left(\frac{n+1}{2}\right)}\,f_l(A)(\eta_l).
\end{equation}
Then the function~$\Phi_l(A)$ is well defined and holomorphic on~$U_l$. Since $f_l(A)\in Z^{n}(\CK_l;\C)$ and $H_n(\CK_l)=0$, we see that the function~$\Phi_l(A)$ is independent of the choice of the chain~$\eta_l$.

\begin{lem}\label{lem_Phi12}
$\Phi_{l_1}(A)=\Phi_{l_2}(A)$ for all\/ $A\in U_{l_1}\cap U_{l_2}$.
\end{lem}
\begin{proof}
Since $A\in U_{l_1}\cap U_{l_2}$, we have $\CK_{l_1}=\CK(A^tA,\varkappa_1)$ and $\CK_{l_1}=\CK(A^tA,\varkappa_2)$, where $\varkappa_1,\varkappa_2<1$. Without loss of generality we may assume that $\varkappa_1>\varkappa_2$. Then $\CK_{l_1}\supseteq\CK_{l_2}$. Hence, we can consider $\eta_{l_2}$ as an element of the group~$C_n(\CK_{l_1})$. Obviously, the restriction of the cochain~$f_{l_1}(A)$ to~$\CK_{l_2}$ coincides with~$f_{l_2}(A)$. %Therefore, $$f_{l_1}(A)(\eta_{l_2})=f_{l_2}(A)(\eta_{l_2})=\Phi_{l_2}(A).$$  
Since $\partial(\eta_{l_1}-\eta_{l_2})=0$ and $H_n(\CK_{l_1})=0$, we obtain that there is a  chain $\zeta\in C_{n+1}(\CK_{l_1})$ such that $\partial\zeta=\eta_{l_1}-\eta_{l_2}$. Then 
$$
f_{l_1}(A)(\eta_{l_1})-f_{l_2}(A)(\eta_{l_2})=f_{l_1}(A)(\partial\zeta)=(-1)^n(\delta f_{l_1}(A))(\zeta)=0.
$$
Thus $\Phi_{l_1}(A)=\Phi_{l_2}(A)$.
\end{proof}
 
Let us proceed with the proof of Proposition~\ref{propos_Phi}. It follows from assertions~(3) and~(4) of Lemma~\ref{lem_KU}, and Lemma~\ref{lem_Phi12} that there exists a holomorphic function $\Phi(A)$ on~$\Sigma_{\C}$ such that $\Phi(A)=\Phi_l(A)$ for all~$A\in U_l$, $l=1,\ldots,L$. 
Suppose that 
$$
\eta_l=\sum_{q=1}^{N_l}c_{l,q}[u_{l,q,0},\ldots,u_{l,q,n}],\qquad c_{l,q}\in\Z.
$$
Substituting this to~\eqref{eq_fl} and~\eqref{eq_Phi_l}, we obtain that
\begin{equation}\label{eq_Phi(A)_l}
\Phi(A)=\frac{2}{\nu^n\Gamma\left(\frac{n+1}2\right)}\sum_{q=1}^{N_l}c_{l,q}F\left(\ba_{u_{l,q,0}},\ldots,\ba_{u_{l,q,n}}\right)
\end{equation}
whenever $A\in U_l$. It follows from~\eqref{eq_or_vol} and~\eqref{eq_orient_vol_F} that the restriction of the function $\Phi(A)$ to~$U_l\cap\Sigma$ coincides with~$\CV_{\xi}(A)$. Since this is true for all~$l$, we see that the restriction of~$\Phi(A)$ to~$\Sigma$ coincides with~$\CV_{\xi}(A)$. 

Finally, notice that the function~$\Phi(A)$ is bounded. Indeed, it follows from~\eqref{eq_est_F} and~\eqref{eq_Phi(A)_l} that 
$$
|\Phi(A)|<\frac{2}{\Gamma\left(\frac{n+1}2\right)}\cdot\left(\frac{\pi^2(n+1)}{4}\right)^{\frac{n+1}4}\cdot
\max\left(\sum_{q=1}^{N_1}|c_{1,q}|,\ldots,\sum_{q=1}^{N_L}|c_{L,q}|\right)
$$  
for all $A\in \Sigma_{\C}$. 
\end{proof}

\begin{proof}[Proof of Theorem~\ref{theorem_main2}]
The function~$\Phi(A)$ in Proposition~\ref{propos_Phi} is a bounded holomorphic function on an affine algebraic variety~$\Sigma_{\C}$. By Corollary~\ref{cor_bhf}, the restriction of~$\Phi(A)$ to every irreducible component of~$\Sigma_{\C}$ is constant. On the other hand, the restriction of~$\Phi(A)$ to the configuration space~$\Sigma=\Sigma(\xi,\X^n,\bell)$ coincides with the function~$\CV_{\xi}(A)$. Therefore the restriction of~$\CV_{\xi}(A)$ to every connected component of~$\Sigma$ is constant. 
\end{proof}

\end{document}